\documentclass[oneside,american,english]{amsart}
\usepackage[T1]{fontenc}
\usepackage[latin9]{inputenc}
\usepackage{babel}
\usepackage{amsthm}
\usepackage{amstext}
\usepackage{amssymb}
\usepackage{esint}
\usepackage[unicode=true,
 bookmarks=true,bookmarksnumbered=false,bookmarksopen=false,
 breaklinks=false,pdfborder={0 0 1},backref=false,colorlinks=false]
 {hyperref}
\usepackage{breakurl}

\makeatletter
\numberwithin{equation}{section}
\numberwithin{figure}{section}
\theoremstyle{plain}
\newtheorem{thm}{\protect\theoremname}
  \theoremstyle{plain}
  \newtheorem{lem}[thm]{\protect\lemmaname}
  \theoremstyle{remark}
  \newtheorem{rem}[thm]{\protect\remarkname}
  \theoremstyle{definition}
  \newtheorem{defn}[thm]{\protect\definitionname}

\makeatother

  \addto\captionsamerican{\renewcommand{\definitionname}{Definition}}
  \addto\captionsamerican{\renewcommand{\lemmaname}{Lemma}}
  \addto\captionsamerican{\renewcommand{\remarkname}{Remark}}
  \addto\captionsamerican{\renewcommand{\theoremname}{Theorem}}
  \addto\captionsenglish{\renewcommand{\definitionname}{Definition}}
  \addto\captionsenglish{\renewcommand{\lemmaname}{Lemma}}
  \addto\captionsenglish{\renewcommand{\remarkname}{Remark}}
  \addto\captionsenglish{\renewcommand{\theoremname}{Theorem}}
  \providecommand{\definitionname}{Definition}
  \providecommand{\lemmaname}{Lemma}
  \providecommand{\remarkname}{Remark}
\providecommand{\theoremname}{Theorem}

\begin{document}

\title{{\small{}Thick Points of High-Dimensional Gaussian Free Fields}}

\author{Linan Chen}

\selectlanguage{american}%

\address{Department of Mathematics and Statistics, McGill University, 805
Sherbrooke Street West, Montréal, QC, H3A 0B9, Canada. }

\email{Email: lnchen@math.mcgill.ca}

\keywords{\noindent Gaussian free field, polynomial singularity, thick point,
Hausdorff dimension}

\subjclass[2000]{\noindent 60G60, 60G15.}

\thanks{The author would like to thank Daniel Stroock for helpful discussions.
The author is partially supported by NSERC Discovery Grant G241023.}
\selectlanguage{english}%
\begin{abstract}
This work aims to extend the existing results on thick points of logarithmic-correlated
Gaussian Free Fields to Gaussian random fields that are more singular.
To be specific, we adopt a sphere averaging regularization to study
polynomial-correlated Gaussian Free Fields in higher-than-two dimensions.
Under this setting, we introduce the definition of thick points which,
heuristically speaking, are points where the value of the Gaussian
Free Field is unusually large. We then establish a result on the Hausdorff
dimension of the sets containing thick points.
\end{abstract}

\maketitle

\section{Introduction}

Many recent developments in statistical physics and probability theory
have seen Gaussian Free Field (GFF) as an indispensable tool. Heuristically
speaking, GFFs are analogues of Brownian motion with multidimensional
time parameters. Just as Brownian motion is thought of as a natural
interpretation of ``random curve'', GFFs are considered as promising
candidates for modeling ``random surface'' or ``random manifold'',
which ultimately lead to the study of random geometry. Motivated by
their importance, GFFs have been widely studied both in discrete and
continuum settings, and certain geometric properties of GFFs have
been revealed. For example, the distribution of extrema and near-extrema
of two-dimensional log-correlated discrete GFFs are studied by Ding
\emph{et al }\cite{DingZeitouni,DingRoyZeitouni,ChatterjeeDemboDing}.
However, for continuum GFFs, the notion of ``extrema'' is not applicable,
because even in the two-dimensional case a generic element of the
GFF is only a tempered distribution which is not defined point-wisely.
In fact, it is the singularity of GFFs that poses most challenges
in obtaining analytic results on the geometry of GFFs. To overcome
most of the challenges, one needs to apply a procedure%
\footnote{In the literature of physics, such a procedure is called a ``regularization''.%
} to approximate point-wise values of GFFs. One such procedure is to
average GFFs over some sufficiently ``nice'' Borel sets. Even though
it is a tempered distribution, a generic element of a GFF can be integrated
over sufficiently regular submanifolds. Using this idea, the notion
of ``thick point''%
\footnote{The term ``thick point'' is borrowed from the literature of stochastic
analysis. There it refers to the extremes of the occupation measure
of a stochastic process. %
} for continuum GFFs, as the analogue of extrema of discrete GFFs,
is introduced and studied by Hu, Miller and Peres in \cite{HMP}.

More specifically, let $h$ be a generic element of the GFF associated
with the operator $\Delta$ on a bounded domain $D\subseteq\mathbb{R}^{2}$
with the Dirichlet boundary condition. Governed by the properties
of the Green's function of $\Delta$ in 2D, such a GFF is logarithmically
correlated, and it is possible to make sense of the circular average
of $h$:
\[
\bar{h}_{t}\left(z\right):=\frac{1}{2\pi t}\int_{\partial B\left(z,t\right)}h\left(x\right)\sigma\left(dx\right)
\]
where $z\in D$, $\partial B\left(z,t\right)$ is the circle centered
at $z$ with radius $t$ and $\sigma\left(dx\right)$ is the length
measure along the circle. To get an approximation of ``$h\left(z\right)$'',
it is to our interest to study $\bar{h}_{t}\left(z\right)$ as $t\searrow0$.
For every $a\geq0$, the set of $a-$thick points of $h$ are defined
in \cite{HMP} as 
\begin{equation}
T_{h}^{a}:=\left\{ z\in D:\,\lim_{t\searrow0}\,\frac{\bar{h}_{t}\left(z\right)}{\left(-\ln t\right)}=\sqrt{\frac{a}{\pi}}\right\} .\label{eq:2D thick point}
\end{equation}
With $z$ fixed, the circular average process $\left\{ \bar{h}_{t}\left(z\right):\, z\in(0,1]\right\} $
has the same distribution as a Brownian motion $\left\{ B_{\tau}\left(z\right):\,\tau\geq0\right\} $
up to a deterministic time change $\tau=\left(-\ln t\right)/\sqrt{2\pi}$,
and as $t\searrow0$, $\bar{h}_{t}\left(z\right)$ behaves just like
$B_{\tau}\left(z\right)$ as $\tau\nearrow\infty$. Then, for any
given $z\in D$, written in terms of $\left\{ B_{\tau}\left(z\right):\,\tau\geq0\right\} $,
the limit involved in (\ref{eq:2D thick point}) is equivalent to
\[
\lim_{\tau\rightarrow\infty}\,\frac{B_{\tau}\left(z\right)}{\tau}=\sqrt{2a}
\]
which occurs with probability zero for any $a>0$. Therefore, $a-$thick
points, so long as $a>0$, are locations where the field value is
``unusually'' large. The authors of \cite{HMP} prove that for every
$a\in\left[0,2\right]$, $\dim_{\mathcal{H}}\left(T_{h}^{a}\right)=2-a$
a.s., where ``$\dim_{\mathcal{H}}$'' denotes the Hausdorff dimension.
Thick points characterize a basic aspect of the ``landscape'' of
GFFs, that is, where the ``high peaks'' occur, and hence thick points
are of importance to understanding the geometry of GFFs. Besides,
the sets containing thick points also arise naturally as supports
of random measures. For example, the Liouville quantum gravity measure
constructed by Duplantier and Sheffield in \cite{DS1} is supported
on a thick point set. Another such example is multiplicative chaos.
In Kahane's paper \cite{Kah}, it is pointed out that multiplicative
chaos lives on a fractal set, which is essentially a thick point set
in a different context. More recently, the results on the support
of multiplicative chaos are reviewed by Rhodes and Vargas in \cite{RV13}.
Through different approximation procedures, the results in \cite{HMP}
are extended by Cipriani and Hazra to more general log-correlated
GFFs (\cite{CiprianiHazra13,CiprianiHazra14}). It is shown that for
log-correlated GFFs in any dimensions, one can similarly define thick
point sets as in (\ref{eq:2D thick point}) and a result on Hausdorff
dimensions of such sets is in order. However, to the best of the author's
knowledge, there had been no comparable study of thick points for
GFFs that are more singular, e.g., polynomial-correlated GFFs. In
fact, to date little is known about the geometry of such GFFs. Inspired
by the approach presented in \cite{HMP}, this article lays out the
first step of an attempt to explore geometric problems associated
with polynomial-correlated GFFs in any dimensions. 

The main focus of this article is to extend the techniques and the
results on thick points of log-correlated GFFs to polynomial-correlated
GFFs on $\mathbb{R}^{\nu}$ for any $\nu>2$. Intuitively speaking,
compared with the log-correlated counterparts, GFFs that are polynomially
correlated consist of generic elements that are more singular so the
``landscape'' of such a field is ``rougher'', and the higher the
dimension $\nu$ is, the worse it becomes. To make these remarks rigorous
and to bring generality to our approach, we adopt the theory of the
Abstract Wiener Space (\cite{aws}) to interpret general Gaussian
random fields, including GFFs with any degree of polynomial singularity
in any dimensions. Let $\theta$ be a generic element of such a field.
It is always possible, by averaging $\theta$ over codimension-1 spheres
centered at $x\in\mathbb{R}^{\nu}$, to obtain a proper approximation
$\bar{\theta}_{t}\left(x\right)$ which approaches ``$\theta\left(x\right)$''
as $t\searrow0$. We give a careful analysis of the two parameter
Gaussian family 
\[
\left\{ \bar{\theta}_{t}\left(x\right):\, x\in\mathbb{R}^{\nu},t\in(0,1]\right\} 
\]
and use the concentric spherical averages (with $x$ fixed) to define
thick points. It turns out that, instead of the most straightforward
analogue of (\ref{eq:2D thick point}), a more suitable definition
of thick point for the degree-$(\nu-2)$-polynomial-correlated GFF
is that, for $\gamma\geq0$, $x$ is a $\gamma-$thick point of $\theta$
if and only if 
\begin{equation}
\limsup_{t\searrow0}\,\frac{\bar{\theta}_{t}\left(x\right)}{\sqrt{-G\left(t\right)\ln t}}\geq\sqrt{2\nu\gamma}\label{eq:thick point def}
\end{equation}
where $G\left(t\right):=\mathbb{E}\left[\left(\bar{\theta}_{t}\left(x\right)\right)^{2}\right]$.
In a similar spirit as (\ref{eq:2D thick point}), if $\gamma>0$,
then a $\gamma-$thick point is a location where $\theta$ is unusually
large. By adapting the approach presented in \cite{HMP}, we establish
the result (Theorem \ref{thm:hausdorff dimension of thick point set})
that, if $T_{\theta}^{\gamma}$ is the set consisting of all the $\gamma-$thick
points of $\theta$ in the unit cube in $\mathbb{R^{\nu}}$, then
\[
\dim_{\mathcal{H}}\left(T_{\theta}^{\gamma}\right)=\nu\left(1-\gamma\right)\;\mbox{a.s..}
\]
Moreover, we investigate the relation between (\ref{eq:2D thick point})
and (\ref{eq:thick point def}), and show that (Theorem \ref{thm: no perfect thick piont})
due to the higher-order singularity of the polynomial-correlated GFFs,
with probability one, the ``perfect'' $\gamma-$thick point, i.e.,
$x$ such that 
\begin{equation}
\lim_{t\searrow0}\,\frac{\bar{\theta}_{t}\left(x\right)}{\sqrt{-G\left(t\right)\ln t}}=\sqrt{2\nu\gamma},\label{eq:perfect thick point def}
\end{equation}
does not exist, which explains why (\ref{eq:thick point def}) is
more suitable a choice than (\ref{eq:2D thick point}) as the definition
of thick point for GFFs that are polynomially correlated. On the other
hand, if we relax the condition in (\ref{eq:perfect thick point def})
to 
\begin{equation}
\lim_{n\rightarrow\infty}\,\frac{\bar{\theta}_{r_{n}}\left(x\right)}{\sqrt{-G\left(r_{n}\right)\ln r_{n}}}=\sqrt{2\nu\gamma},\label{eq:thick point along sequence}
\end{equation}
where $\left\{ r_{n}:\, n\geq0\right\} $ is any sequence that decays
to zero sufficiently fast, then we find out (Theorem \ref{thm:thick point along sequence})
that, if $ST_{\theta}^{\gamma}$ is the set consisting of all the
points $x$ in the unit cube in $\mathbb{R^{\nu}}$ that satisfies
(\ref{eq:thick point along sequence}), then 
\[
\dim_{\mathcal{H}}\left(ST_{\theta}^{\gamma}\right)=\nu\left(1-\gamma\right)\;\mbox{a.s..}
\]
Some lemmas we obtained during the process are of independent interest. 

In $\mathsection2$ we briefly introduce the theory of the Abstract
Wiener Space as the foundation for the study of GFFs. In $\mathsection3$
we give a detailed study of the Gaussian family consisting of spherical
averages of the GFFs. These are the main tools that will be exploited
in later parts of this article. Our main results are stated in $\mathsection4$
and at the beginning of $\mathsection5$. In particular, the result
on $\dim_{\mathcal{H}}\left(T_{\theta}^{\gamma}\right)$ is proved
by establishing the upper bound and the lower bound separately. The
upper bound is proved in $\mathsection4.1$, and the lower bound is
established in multiple steps in $\mathsection5$.

\section{Abstract Wiener Space and Gaussian Free Fields}

\selectlanguage{american}%
The theory of Abstract Wiener Space (AWS), first introduced by Gross
\cite{aws}, provides an analytical foundation for the construction
and the study of Gaussian measures in infinite dimensions. To be specific,
given a real separable Banach space $E$, a non-degenerate centered
Gaussian measure $\mathcal{W}$ on $E$ is a Borel probability measure
such that for every $x^{*}\in E^{*}\backslash\{0\}$, the functional
$x\in E\mapsto\left\langle x,x^{*}\right\rangle \in\mathbb{R}$ has
non-degenerate centered Gaussian distribution under $\mathcal{W}$,
where $E^{*}$ is the space of bounded linear functionals on $E$,
and $\left\langle \cdot,x^{*}\right\rangle $ is the action of $x^{*}\in E^{*}$
on $E$. Further assume that $H$ is a real separable Hilbert space
which is continuously embedded in $E$ as a dense subspace. Then $E^{*}$
can be continuously and densely embedded into $H$, and for every
$x^{*}\in E^{*}$ there exists a unique $h_{x^{*}}\in H$ such that
$\left\langle h,x^{*}\right\rangle =\left(h,h_{x^{*}}\right)_{H}$
for all $h\in H$. Under this setting if the Gaussian measure $\mathcal{W}$
on $E$ has the following characteristic function: 
\[
\mathbb{E}^{\mathcal{W}}\left[\exp\left(i\left\langle \cdot,x^{*}\right\rangle \right)\right]=\exp\left(-\frac{\left\Vert h_{x^{*}}\right\Vert _{H}^{2}}{2}\right)\mbox{ for every }x^{*}\in E^{*},
\]
then the triple $\left(H,E,\mathcal{W}\right)$ is called an \emph{Abstract
Wiener Space}. Moreover, since $\left\{ h_{x^{*}}:\, x^{*}\in E^{*}\right\} $
is dense in $H$, the mapping 
\[
\mathcal{I}:\, h_{x^{*}}\in H\mapsto\mathcal{I}\left(h_{x^{*}}\right):=\left\langle \cdot,x^{*}\right\rangle \in L^{2}\left(\mathcal{W}\right)
\]
can be uniquely extended to a linear isometry between $H$ and $L^{2}\left(\mathcal{W}\right)$.
The extended isometry, also denoted by $\mathcal{I}$, is the \emph{Paley-Wiener
map} and its images $\left\{ \mathcal{I}\left(h\right):\, h\in H\right\} $,
known as the \emph{Paley-Wiener integrals}, form a centered Gaussian
family whose covariance is given by 
\[
\mathbb{E}^{\mathcal{W}}\left[\mathcal{I}\left(h\right)\mathcal{I}\left(g\right)\right]=\left(h,g\right)_{H}\mbox{ for all }h,g\in H.
\]
Therefore, if $\left\{ h_{n}:n\geq1\right\} $ is an orthonormal basis
of $H$, then $\left\{ \mathcal{I}\left(h_{n}\right):n\geq1\right\} $
is a family of i.i.d. standard Gaussian random variables. In fact,
\begin{equation}
\mbox{ for }\mathcal{W}-\mbox{a.e. }x\in E,\quad x=\sum_{n\geq1}\mathcal{I}\left(h_{n}\right)\left(x\right)h_{n}.\label{eq:H_basis expansion}
\end{equation}

Although $\mathcal{W}$ is a measure on $E$, it is the inner product
of $H$ that determines the covariance structure of $\mathcal{W}$.
$H$ is referred to as the \emph{Cameron-Martin space} of $\left(H,E,\mathcal{W}\right)$.
The theory of AWS says that given any separable Hilbert space $H$,
one can always find $E$ and $\mathcal{W}$ such that the triple $\left(H,E,\mathcal{W}\right)$
forms an AWS. On the other hand, given a separable Banach space $E$,
any non-degenerate centered Gaussian measure $\mathcal{W}$ on $E$
must exist in the form of an AWS. That is to say that, AWS is the
``natural'' format in which any infinite dimensional Gaussian measure
exists. For further discussions on the construction and the properties
of AWS, we refer to \cite{aws}, \cite{awsrevisited}, \cite{add_Gaus}
and $\mathsection8$ of \cite{probability}.

\selectlanguage{english}%
We now apply the general theory of AWS to study Gaussian measures
on function or generalized function spaces. To be specific, given
$\nu\in\mathbb{N}$ and $p\in\mathbb{R}$, consider the Sobolev space
$H^{p}:=H^{p}\left(\mathbb{R}^{\nu}\right)$, which is the closure
of $C_{c}^{\infty}\left(\mathbb{R}^{\nu}\right)$, the space of compactly
supported smooth functions on $\mathbb{R}^{\nu}$, under the inner
product given by, for $\phi,\,\psi\in C_{c}^{\infty}\left(\mathbb{R}^{\nu}\right)$,
\begin{eqnarray*}
\left(\phi,\,\psi\right)_{_{H^{p}}} & := & \left(\left(I-\Delta\right)^{p}\phi,\psi\right)_{L^{2}\left(\mathbb{R}^{\nu}\right)}\\
 & = & \frac{1}{\left(2\pi\right)^{\nu}}\int_{\mathbb{R}^{\nu}}\left(1+\left|\xi\right|^{2}\right)^{p}\hat{\phi}\left(\xi\right)\overline{\hat{\psi}\left(\xi\right)}d\xi,
\end{eqnarray*}
where $\hat{\cdot}$ denotes the Fourier transform. $\left(H^{p},\,\left(\cdot,\cdot\right)_{H^{p}}\right)$
is a separable Hilbert space, and it will be taken as the Cameron-Martin
space for the discussions in this article. As mentioned earlier, there
exists a separable Banach space $\Theta^{p}:=\Theta^{p}\left(\mathbb{R}^{\nu}\right)$
and a Gaussian measure $\mathcal{W}^{p}:=\mathcal{W}^{p}\left(\mathbb{R}^{\nu}\right)$
on $\Theta^{p}$ such that the triple $\left(H^{p},\Theta^{p},\mathcal{W}^{p}\right)$
forms an AWS, to which we refer as the \emph{dim-$\nu$ order-$p$}
\emph{Gaussian Free Field }(GFF) %
\footnote{In physics literature, the term ``GFF'' only refers to the case
when $p=1$. Here we slightly extend the use of this term and continue
to use GFF.%
}.\emph{ }It is clear that the covariance of such a field is determined
by the Green's function of $\left(I-\Delta\right)^{p}$ on $\mathbb{R}^{\nu}$. 

To give explicit formulations for the GFFs introduced in the framework
above, we review the result in \foreignlanguage{american}{\emph{\cite{probability}}}
($\mathsection8.5$) that, when $p=\frac{\nu+1}{2}$, $\Theta^{\frac{\nu+1}{2}}$
can be taken as 
\[
\Theta^{\frac{\nu+1}{2}}:=\left\{ \theta\in C\left(\mathbb{R}^{\nu}\right):\lim_{\left|x\right|\rightarrow\infty}\,\frac{\left|\theta\left(x\right)\right|}{\log\left(e+\left|x\right|\right)}=0\right\} ,
\]
equipped with the norm 
\[
\left\Vert \theta\right\Vert _{\Theta^{\frac{\nu+1}{2}}}:=\sup_{x\in\mathbb{R}^{\nu}}\,\frac{\left|\theta\left(x\right)\right|}{\log\left(e+\left|x\right|\right)}.
\]
In other words, the dim-$\nu$ order-$\frac{\nu+1}{2}$ GFF consists
of continuous functions on $\mathbb{R}^{\nu}$. More generally, for
$p\in\mathbb{R}$, $H^{p}$ is the isometric image of $H^{\frac{\nu+1}{2}}$
under the Bessel-type operator $\left(I-\Delta\right)^{\frac{\nu+1-2p}{4}}$.
Therefore, we can take $\Theta^{p}$ to be the image of $\Theta^{\frac{\nu+1}{2}}$
under $\left(I-\Delta\right)^{\frac{\nu+1-2p}{4}}$ and the corresponding
Gaussian measure is 
\[
\mathcal{W}^{p}=\left(\left(I-\Delta\right)^{-\frac{\nu+1-2p}{4}}\right)_{\star}\mathcal{W}^{\frac{\nu+1}{2}}.
\]
In addition, if we identify $H^{-p}$ as the dual space of $H^{p}$,
then $\left(\Theta^{p}\right)^{*}\subseteq H^{-p}$ and for every
$\lambda\in\left(\Theta^{p}\right)^{*}$, it is easy to see that 
\begin{equation}
\lambda\mapsto h_{\lambda}:=\left(I-\Delta\right)^{-p}\lambda\label{eq:(1-Delta)^(-p)}
\end{equation}
gives the unique element $h_{\lambda}\in H^{p}$ such that the action
of $\lambda\in\left(\Theta^{p}\right)^{*}$, when restricted on $H^{p}$,
coincides with $\left(\cdot,h_{\lambda}\right)_{H^{p}}$. Moreover,
the map (\ref{eq:(1-Delta)^(-p)}) can also be viewed as an isometry
between $H^{-p}$ and $H^{p}$. For $\lambda\in H^{-p}$, we still
use ``$h_{\lambda}$'' to denote the image of $\lambda$ under (\ref{eq:(1-Delta)^(-p)}).
Then the Paley-Wiener integrals $\left\{ \mathcal{I}\left(h_{\lambda}\right):\lambda\in H^{-p}\right\} $
form a centered Gaussian family with the covariance 
\[
\mathbb{E}^{\mathcal{W}^{p}}\left[\mathcal{I}\left(h_{\lambda}\right)\mathcal{I}\left(h_{\eta}\right)\right]=\left(h_{\lambda},h_{\eta}\right)_{H^{p}}=\left(\lambda,\eta\right)_{H^{-p}}\mbox{ for every }\lambda,\eta\in H^{-p}.
\]

It is clear from the discussions above that with the dimension $\nu$
fixed, the larger the order $p$ is, the more regular the elements
of the GFF are; on the other hand, if $p$ is fixed, then the higher
the dimension $\nu$ is, the more singular the GFF becomes. In most
of the cases that are of interest to us, generic elements of GFFs
are only tempered distributions. For example, this is the case with
GFFs that are logarithmically correlated. Interpreted under the framework
introduced above, log-correlated GFFs are dim-$\nu$ order-$(\nu/2)$
GFFs, i.e., with $p=\nu/2$, since the Green's function of $\left(I-\Delta\right)^{\nu/2}$
on $\mathbb{\mathbb{R}^{\nu}}$ has logarithmic singularity along
the diagonal. On the other hand, if $2p\in\mathbb{N}$ and $2p<\nu$,
the Green's function have polynomial singularity with degree $\nu-2p$
and hence the corresponding GFFs are polynomially correlated. In this
article, we focus on studying certain geometric properties of polynomial-correlated
GFFs with%
\footnote{The GFFs with $p$ being half integers (and hence the operator is
non-local) are considered by O. Nadeau-Chamard and the author in a
separate paper which is currently in preparation.%
} $p\in\mathbb{N}$ and $p<\nu/2$.

We finish this section by remarking that instead of using the Bessel-type
operator $\left(I-\Delta\right)^{p}$ to construct GFFs on $\mathbb{R}^{\nu}$,
one can also use the operator $\Delta^{p}$, equipped with proper
boundary conditions, to construct GFFs on bounded domains on $\mathbb{R}^{\nu}$(e.g.,
\cite{HMP}, \cite{DS1} and \cite{Shef}). The field elements obtained
in either way possess similar local properties. However, $\left(I-\Delta\right)^{p}$
rather than $\Delta^{p}$ is a better choice for this project for
the technical reason that $\left(I-\Delta\right)^{p}$ allows the
GFF to be defined on the entire space, and hence we do not have to
specify a boundary condition, which is an advantage at least when
$p>1$.

\section{Spherical Averages of Gaussian Free Fields}

For the rest of this article, we assume that $\nu,p\in\mathbb{N}$,
$\nu>2$ and $1\leq p<\nu/2$, and $\theta$ is a generic element
of the dim-$\nu$ order-$p$ GFF, i.e., $\theta\in\Theta^{p}$ is
sampled under $\mathcal{W}^{p}$. Although ``$\theta\left(x\right)$''
is not defined for every $x\in\mathbb{R}^{\nu}$, we can use the ``average''
of $\theta$ over a sphere centered at $x$ to approximate ``$\theta\left(x\right)$'',
as the radius of the sphere tends to zero. To make this precise, we
need to introduce some notation. Let $B\left(x,t\right)$ and $\partial B\left(x,t\right)$
be the open ball and, respectively, the sphere centered at $x\in\mathbb{R}^{\nu}$
with radius (under the Euclidean metric) $t>0$, $\sigma_{x,t}$ the
surface measure on $\partial B\left(x,t\right)$, $\alpha_{\nu}\left(t\right):=\alpha_{\nu}t^{\nu-1}$
the surface area of $\partial B\left(x,t\right)$ with $\alpha_{\nu}:=2\pi^{\nu/2}/\Gamma\left(\nu/2\right)$,
and $\sigma_{x,t}^{ave}:=\sigma_{x,t}/\alpha_{\nu}\left(t\right)$
the spherical average measure over $\partial B\left(x,t\right)$.
We first state the following simple facts about $\sigma_{x,t}^{ave}$.
It is straightforward to derive these results, so we will omit the
proofs.
\begin{lem}
For every $x\in\mathbb{R}^{\nu}$ and $t>0$, $\sigma_{x,t}^{ave}\in H^{-1}\left(\mathbb{R}^{\nu}\right)$
and its Fourier transform is given by 
\begin{equation}
\forall\xi\in\mathbb{R}^{\nu},\quad\widehat{\sigma_{x,t}^{ave}}\left(\xi\right)=\frac{\left(2\pi\right)^{\frac{\nu}{2}}}{\alpha_{\nu}}\, e^{i\left(x,\xi\right)_{\mathbb{R}^{\nu}}}\cdot\left(t\left|\xi\right|\right)^{\frac{2-\nu}{2}}J_{\frac{\nu-2}{2}}\left(t\left|\xi\right|\right)\label{eq:Fourier transform of spherical average}
\end{equation}
where $J_{\frac{\nu-2}{2}}$ is the standard Bessel function of the
first kind with index $\frac{\nu-2}{2}$.
\end{lem}
The first assertion of the lemma implies that $\sigma_{x,t}^{ave}\in H^{-p}\left(\mathbb{R}^{\nu}\right)$
for every $p\geq1$. In particular, this fact shows that, no matter
what the dimension is and how singular the GFF is, a codimension-1
sphere is always sufficiently ``nice'' that it is possible to average
the GFF over such a sphere. As a consequence, $\mathcal{I}\left(h_{\sigma_{x,t}^{ave}}\right)$,
viewed as the spherical average of the GFF, is well defined for every
$x\in\mathbb{R}^{\nu}$ and $t>0$ as a Gaussian random variable,
and as $t\searrow0$, from the point of the view of tempered distributions,
$\mathcal{I}\left(h_{\sigma_{x,t}^{ave}}\right)\left(\theta\right)$
approximates ``$\theta\left(x\right)$''. With the help of (\ref{eq:Fourier transform of spherical average}),
we can compute, by Parseval's identity, the covariance of the Gaussian
family consisting of all the spherical averages and express the covariance
as follows. 
\begin{lem}
$\left\{ \mathcal{I}\left(h_{\sigma_{x,t}^{ave}}\right):\, x\in\mathbb{R}^{\nu},\, t>0\right\} $
is a two-parameter centered Gaussian family under $\mathcal{W}^{p}$,
and the covariance is given by, for $x,y\in\mathbb{R}^{\nu}$ and
$t,s>0$, 
\begin{eqnarray}
 &  & \begin{split} & \mathbb{E}^{\mathcal{W}^{p}}\left[\mathcal{I}\left(h_{\sigma_{x,t}^{ave}}\right)\mathcal{I}\left(h_{\sigma_{y,s}^{ave}}\right)\right]\\
 & \qquad=\frac{\left(2\pi\right)^{\nu/2}}{\alpha_{\nu}^{2}\left(ts\left|x-y\right|\right)^{\frac{\nu-2}{2}}}\int_{0}^{\infty}\frac{\tau^{2-\frac{\nu}{2}}J_{\frac{\nu-2}{2}}\left(t\tau\right)J_{\frac{\nu-2}{2}}\left(s\tau\right)J_{\frac{\nu-2}{2}}\left(\left|x-y\right|\tau\right)}{\left(1+\tau^{2}\right)^{p}}d\tau.
\end{split}
\label{eq:covariance for (1-Delta)^s-1}
\end{eqnarray}
In particular, when $x=y$, i.e., in the case of concentric spherical
averages,
\begin{equation}
\mathbb{E}^{\mathcal{W}^{p}}\left[\mathcal{I}\left(h_{\sigma_{x,t}^{ave}}\right)\mathcal{I}\left(h_{\sigma_{x,s}^{ave}}\right)\right]=\frac{1}{\alpha_{\nu}\left(ts\right)^{\frac{\nu-2}{2}}}\int_{0}^{\infty}\frac{\tau J_{\frac{\nu-2}{2}}\left(t\tau\right)J_{\frac{\nu-2}{2}}\left(s\tau\right)}{\left(1+\tau^{2}\right)^{p}}d\tau.\label{eq:covariance for (1-Delta)^s concentric}
\end{equation}

\end{lem}
Again, these results follow easily from integral representations of
Bessel functions (\cite{BesselFunctions}, $\mathsection3.3$) combined
with straightforward computations. Proofs are omitted.

To study the distribution of the family of spherical averages, and
to use them effectively to approximate ``pointwise values'' of the
GFF, it is useful to obtain as explicit an expression for the covariance
as possible. To this end, we will first assume $p=1$ and treat the
concentric spherical averages ($\mathsection3.1$) and the non-concentric
ones ($\mathsection3.2$) separately. During this process, we find
for each $x\in\mathbb{R}^{\nu}$ a set of ``renormalized spherical
averages'' which still approximates ``$\theta\left(x\right)$''
but whose covariance has technically desirable properties. In $\mathsection3.3$
we briefly explain the strategy for treating the spherical averages
when $p>1$.

\subsection{When $p=1$. Concentric Spherical Averages.}

For the rest of this article, when $p=1$, we simply write $\left(H^{p},\Theta^{p},\mathcal{\mathcal{W}}^{p}\right)$
as $\left(H,\Theta,\mathcal{\mathcal{W}}\right)$. It is clear from
(\ref{eq:covariance for (1-Delta)^s concentric}) that the distribution
of the concentric spherical averages $\left\{ \mathcal{I}\left(h_{\sigma_{x,t}^{ave}}\right):\, t>0\right\} $
at any given $x\in\mathbb{R}^{\nu}$ is independent of $x$. In fact,
the distribution of the GFF is translation invariant. First we state
a closed formula for the integral in (\ref{eq:covariance for (1-Delta)^s concentric}). 
\begin{lem}
\label{lem:covariance C^1(t,s)} Fix any $x\in\mathbb{R}^{\nu}$.
For every $t,s>0$, 
\begin{equation}
\mathbb{E}^{\mathcal{W}}\left[\mathcal{I}\left(h_{\sigma_{x,t}^{ave}}\right)\mathcal{I}\left(h_{\sigma_{x,s}^{ave}}\right)\right]=\frac{1}{\alpha_{\nu}\left(ts\right)^{\frac{\nu-2}{2}}}I_{\frac{\nu-2}{2}}\left(t\wedge s\right)K_{\frac{\nu-2}{2}}\left(t\vee s\right),\label{eq:concentric cov before renorm}
\end{equation}
where $I_{\frac{\nu-2}{2}}$ and $K_{\frac{\nu-2}{2}}$ are the modified
Bessel functions (with pure imaginary argument) with the index $\frac{\nu-2}{2}$.
\end{lem}
One can use a formula in \cite{BesselFunctions} ($\mathsection$13.53)
to derive (\ref{eq:concentric cov before renorm}) directly. An alternative
proof was provided in the Appendix of \cite{CJ}. So we will omit
the proof of Lemma \ref{lem:covariance C^1(t,s)} and refer to \cite{BesselFunctions}
and \cite{CJ} for details. 

By (\ref{eq:concentric cov before renorm}), $\left\{ \mathcal{I}\left(h_{\sigma_{x,t}^{ave}}\right):\, t>0\right\} $
is a backward%
\footnote{A ``backward'' Markov process is a process which exhibits Markov
property as the parameter ``$t$'' decreases. %
} Markov Gaussian process. In fact, (\ref{eq:concentric cov before renorm})
leads to a renormalization of the spherical averages, i.e., 
\[
\bar{\sigma}_{x,t}:=\frac{\left(t/2\right)^{\frac{\nu-2}{2}}}{\Gamma\left(\nu/2\right)\cdot I_{\frac{\nu-2}{2}}\left(t\right)}\cdot\sigma_{x,t}^{ave}.
\]
Denote by $\bar{\theta}_{t}\left(x\right)$ the corresponding Paley-Wiener
integral $\mathcal{I}\left(h_{\bar{\sigma}_{x,t}}\right)\left(\theta\right)$.
Because 
\[
\lim_{t\rightarrow0}\frac{\left(t/2\right)^{\frac{\nu-2}{2}}}{\Gamma\left(\nu/2\right)\cdot I_{\frac{\nu-2}{2}}\left(t\right)}=1,
\]
$\bar{\theta}_{t}\left(x\right)$ still is a legitimate approximation
of ``$\theta\left(x\right)$'' as $t\searrow0$. It follows from
(\ref{eq:concentric cov before renorm}) that the covariance of the
Gaussian process $\left\{ \bar{\theta}_{t}\left(x\right):\, t>0\right\} $
is given by, for $0<s\leq t$,
\begin{equation}
\mathbb{E}^{\mathcal{W}}\left[\bar{\theta}_{t}\left(x\right)\bar{\theta}_{s}\left(x\right)\right]=\frac{\alpha_{\nu}}{\left(2\pi\right)^{\nu}}\cdot\frac{K_{\frac{\nu-2}{2}}\left(t\right)}{I_{\frac{\nu-2}{2}}\left(t\right)}=:G\left(t\right).\label{eq:concentric cov renorm}
\end{equation}
The function $G$ defined above is positive and decreasing on $\left(0,\infty\right)$,
and when $t$ is sufficiently small, $G\left(t\right)=\mathcal{O}\left(t^{2-\nu}\right)$,
which reflects the fact that the dim-$\nu$ order-1 GFF is polynomially
correlated with degree $\nu-2$. 
\begin{rem}
Since we are only concerned about $\bar{\theta}_{t}\left(x\right)$
when $t$ is small, without loss of generality, we assume that $t\in(0,1]$.
As a consequence of (\ref{eq:concentric cov renorm}), $\left\{ \bar{\theta}_{t}\left(x\right):\, t\in(0,1]\right\} $
is a Gaussian process with independent increment (in the direction
of $t$ decreasing), which, up to a time change, has the same distribution
as a Brownian motion. To be specific, if we define a ``clock'' by
\[
\tau:=G\left(t\right)-G\left(1\right)\mbox{ for }t\in(0,1],
\]
then 
\[
\left\{ B_{\tau}:=\bar{\theta}_{G^{-1}\left(\tau+G\left(1\right)\right)}\left(x\right)-\bar{\theta}_{1}\left(x\right):\:\tau\geq0\right\} 
\]
has the same distribution as a standard Brownian motion. 
\end{rem}
Based on the preceding observations, results about the Brownian motion
can be transported directly to $\left\{ \bar{\theta}_{t}\left(x\right):\, t\in(0,1]\right\} $,
and the behavior of $\bar{\theta}_{t}\left(x\right)$ when $t$ is
small resembles that of the Brownian motion $B_{\tau}$ when $\tau$
is large. For example, by the law of the iterated logarithm, 
\begin{equation}
\limsup_{t\searrow0}\frac{\left|\bar{\theta}_{t}\left(x\right)\right|}{\sqrt{2G\left(t\right)\cdot\ln\ln G\left(t\right)}}=1\mbox{ a.s.}.\label{eq:LIL for concentric}
\end{equation}

\subsection{When $p=1$. Non-concentric Spherical Averages.}

We now move on to the family of non-concentric spherical averages.
Again, instead of the regular spherical averages, we adopt the renormalized
spherical averages introduced in $\mathsection3.1$. Consider the
two-parameter Gaussian family 
\[
\left\{ \bar{\theta}_{t}\left(x\right):\, x\in\mathbb{R}^{\nu},\, t\in(0,1]\right\} ,
\]
and denote by $\mbox{Cov}\left(x,t;\, y,s\right)$ the covariance
of $\bar{\theta}_{t}\left(x\right)$ and $\bar{\theta}_{s}\left(y\right)$
for $x,y\in\mathbb{R}^{\nu}$ and $t,s\in(0,1]$. One can compute
$\mbox{Cov}\left(x,t;\, y,s\right)$ using (\ref{eq:covariance for (1-Delta)^s-1})
and the renormalization. In fact, under certain circumstances, it
is possible to obtain explicit formulas for $\mbox{Cov}\left(x,t;\, y,s\right)$. 
\begin{lem}
\label{lem:asymptotics of the cov functions}Let $x,y\in\mathbb{R}^{\nu}$and
$t,s\in(0,1]$.\\
(i) If $\left|x-y\right|\geq t+s$, i.e., if $B\left(x,t\right)\cap B\left(y,s\right)=\emptyset$,
\begin{equation}
\mbox{Cov}\left(x,t;\, y,s\right)=\left(2\pi\right)^{-\nu/2}\cdot\frac{K_{\frac{\nu-2}{2}}\left(\left|x-y\right|\right)}{\left|x-y\right|^{\frac{\nu-2}{2}}}=:C_{disj}\left(\left|x-y\right|\right),\label{eq:cov non-concentric non-overlapping}
\end{equation}
In particular, $C_{disj}\left(\left|x-y\right|\right)=\mathcal{O}\left(\left|x-y\right|^{2-\nu}\right)$
when $\left|x-y\right|$ is small.\\
(ii) If $t\geq\left|x-y\right|+s$, i.e., if $B\left(x,t\right)\supset B\left(y,s\right)$,
\begin{equation}
\mbox{Cov}\left(x,t;\, y,s\right)=\left(2\pi\right)^{-\nu/2}\cdot\frac{I_{\frac{\nu-2}{2}}\left(\left|x-y\right|\right)}{\left|x-y\right|^{\frac{\nu-2}{2}}}\cdot\frac{K_{\frac{\nu-2}{2}}\left(t\right)}{I_{\frac{\nu-2}{2}}\left(t\right)}=:C_{incl}\left(t,\left|x-y\right|\right),\label{eq:cov non-concentric inclusion}
\end{equation}
In particular, $C_{incl}\left(t,\left|x-y\right|\right)=\mathcal{O}\left(t^{2-\nu}\right)$
when $t$ is small. 
\end{lem}
Again, by combining (\ref{eq:covariance for (1-Delta)^s-1}) with
a formula in \cite{BesselFunctions} ($\mathsection$13.53\foreignlanguage{american}{,
pp 429-430}), one can easily verify these results. An alternative
derivation was also provided in the Appendix of \cite{CJ}. We omit
the proofs and refer to \cite{BesselFunctions} and \cite{CJ} for
details. We remark that (\ref{eq:cov non-concentric non-overlapping})
and (\ref{eq:cov non-concentric inclusion}) demonstrate the advantage
of this particular renormalization of the spherical averages. For
the family of the renormalized spherical averages, under the hypothesis
(i) or (ii) in Lemma \ref{lem:asymptotics of the cov functions},
small radius (radii) does not affect the covariance, which favors
convergence as radius (radii) tends to zero. 

However, one still needs to treat the renormalized spherical averages
in the most general case. To this end, we introduce the intrinsic
metric $d$ associated with the Gaussian family $\left\{ \bar{\theta}_{t}\left(x\right):\, x\in\mathbb{R}^{\nu},\, t\in(0,1]\right\} $
where
\[
d\left(x,t;\, y,s\right):=\left(\mathbb{E}^{\mathcal{W}}\left[\left|\bar{\theta}_{t}\left(x\right)-\bar{\theta}_{s}\left(y\right)\right|^{2}\right]\right)^{\frac{1}{2}}
\]
for $x,y\in\mathbb{R}^{\nu}$ and $t,s\in(0,1]$. Assuming $0<s\leq t\leq1$,
the triangle inequality implies that 
\begin{equation}
d\left(x,t;\, y,s\right)\leq d\left(x,t;\, y,t\right)+\sqrt{G\left(s\right)-G\left(t\right)},\label{eq:triangle for d(x,t;y,s)}
\end{equation}
so to work with $d\left(x,t;\, y,s\right)$, we need to study $d\left(x,t;\, y,t\right)$,
i.e., the intrinsic metric associated with the family $\left\{ \bar{\theta}_{t}\left(x\right):\, x\in\mathbb{R}^{\nu}\right\} $
with $t\in(0,1]$ fixed. 
\begin{lem}
\label{lem:estimate for d^2_t non-concentric} There exists a constant%
\footnote{Throughout the article, $C_{\nu}$ denotes a constant that only depends
on the dimension, and $C_{\nu}$'s value may vary from line to line.%
} $C_{\nu}>0$ such that for every $t\in(0,1]$ and every $x,y\in\mathbb{R}^{\nu}$,
\begin{equation}
d^{2}\left(x,t;\, y,t\right)\leq C_{\nu}\cdot t^{2-\nu}\left(\sqrt{\frac{\left|x-y\right|}{t}}\wedge1\right).\label{eq:estimate intrinsic metric non-concentric}
\end{equation}
\end{lem}
\begin{proof}
Based on (\ref{eq:cov non-concentric non-overlapping}), when $\left|x-y\right|\geq2t$,
\[
d^{2}\left(x,t;\, y,t\right)=2G\left(t\right)-2C_{disj}\left(\left|x-y\right|\right)
\]
which immediately implies (\ref{eq:estimate intrinsic metric non-concentric}).
More generally, using (\ref{eq:covariance for (1-Delta)^s-1}) and
(\ref{eq:covariance for (1-Delta)^s concentric}), we can rewrite
$d^{2}\left(x,t;\, y,t\right)$ as 
\[
\begin{split}d^{2}\left(x,t;\, y,t\right) & =\mathbb{E}^{\mathcal{W}}\left[\left(\bar{\theta}_{t}\left(x\right)-\bar{\theta}_{t}\left(y\right)\right)^{2}\right]\\
 & =\frac{2\alpha_{\nu}}{\left(2\pi\right)^{\nu}I_{\frac{\nu-2}{2}}^{2}\left(t\right)}\int_{0}^{\infty}\frac{\tau}{1+\tau^{2}}J_{\frac{\nu-2}{2}}^{2}\left(t\tau\right)\Psi\left(\tau\left|x-y\right|\right)d\tau
\end{split}
\]
 where $\Psi$ is the function given by 
\[
\forall w\in\left(0,\infty\right),\quad\Psi\left(w\right):=1-\frac{\left(2\pi\right)^{\nu/2}}{\alpha_{\nu}}w^{\frac{2-\nu}{2}}J_{\frac{\nu-2}{2}}\left(w\right).
\]
It follows from the properties of $J_{\frac{\nu-2}{2}}$ that $\Psi$
is analytic and 
\[
\Psi\left(w\right)=\Gamma\left(\nu/2\right)\sum_{m=1}^{\infty}\frac{\left(-1\right)^{m-1}2^{-2m}}{m!\Gamma\left(\frac{\nu}{2}+m\right)}\cdot w^{2m}.
\]
Clearly, there exists $C_{\nu}>0$ such that $\left|\Psi\left(w\right)\right|\leq C_{\nu}\sqrt{w}$
for all $w\in[0,\infty)$. Therefore, 
\[
d^{2}\left(x,t;\, y,t\right)\leq C_{\nu}\cdot t^{2-\nu}\sqrt{\left|x-y\right|}\int_{0}^{\infty}\frac{\tau^{3/2}}{1+\tau^{2}}J_{\frac{\nu-2}{2}}^{2}\left(t\tau\right)d\tau,
\]
and the integral on the right, after a change of variable $u=t\tau$,
becomes 
\[
t^{-1/2}\int_{0}^{\infty}\frac{u^{3/2}}{t^{2}+u^{2}}J_{\frac{\nu-2}{2}}^{2}\left(u\right)du\leq t^{-1/2}\int_{0}^{\infty}u^{-1/2}J_{\frac{\nu-2}{2}}^{2}\left(u\right)du=C_{\nu}\cdot t^{-1/2}
\]
which leads to the desired inequality.
\end{proof}
Based on (\ref{eq:triangle for d(x,t;y,s)}) and (\ref{eq:estimate intrinsic metric non-concentric}),
it follows from the Kolmogorov continuity theorem that there exists
a continuous modification of $\left\{ \bar{\theta}_{t}\left(x\right):\, x\in\mathbb{R}^{\nu},\, t\in(0,1]\right\} $.
From now on, we assume that $\left\{ \bar{\theta}_{t}\left(x\right):\, x\in\mathbb{R}^{\nu},\, t\in(0,1]\right\} $
is such a modification. In other words, we assume that for every $\theta\in\Theta$,
$\left(x,t\right)\in\mathbb{R}^{\nu}\times(0,1]\mapsto\bar{\theta}_{t}\left(x\right)\in\mathbb{R}$
is continuous.

Since the distribution of the GFF is translation invariant and the
notion of ``thick point'' only concerns local properties of the
GFF, without loss of generality, we may restrict the GFF to $\overline{S\left(O,1\right)}$
the closed cube centered at the origin with side length $2$ under
the Euclidean metric%
\footnote{Similarly, for $x\in\overline{S\left(O,1\right)}$ and $s>0$, $\overline{S\left(x,s\right)}$
is the Euclidean closed cube centered at $x$ with side length $2s$.%
}. We will apply the metric entropy method (\cite{AT,Talagrand,Dudley})
to study the boundedness and the continuity of the family $\left\{ \bar{\theta}_{t}\left(x\right):\, x\in\overline{S\left(O,1\right)},\, t\in(0,1]\right\} $.
To set this up, we need to introduce some more notation. For every
compact subset $A\subseteq\overline{S\left(O,1\right)}\times(0,1]$,
let $\mbox{diam}_{d}\left(A\right)$ be the diameter of $A$ under
the metric $d$. $A$ is also compact under $d$, so $A$ can be finitely
covered under $d$. For $\epsilon>0$ and $\mathbf{x}\in\overline{S\left(O,1\right)}\times(0,1]$,
let $B_{d}\left(\mathbf{x},\epsilon\right)$ be the open ball centered
at $\mathbf{x}$ with radius $\epsilon$ under $d$, and $N\left(\epsilon,A\right)$
the smallest number of such balls $B_{d}\left(\mathbf{x},\epsilon\right)$
required to cover $A$. Then $N$ is the metric entropy function with
respect to $d$. Applying the standard entropy methods, we get the
following results.
\begin{lem}
\label{lem:expectation of max non-concentric}There exists a constant
$C_{\nu}>0$ such that for every $t,s\in(0,1]$ and every $x\in\overline{S\left(O,1\right)}$,
\begin{equation}
\mathbb{E}^{\mathcal{W}}\left[\sup_{y\in\overline{S\left(x,s\right)}}\,\left|\bar{\theta}_{t}\left(y\right)\right|\right]\leq C_{\nu}\cdot t^{1-\nu/2}\left(\frac{s^{1/4}}{t^{1/4}}\wedge\sqrt{\ln\left(\frac{s}{t}\right)}\right).\label{eq:expectation of sup spatial}
\end{equation}
\end{lem}
\begin{proof}
By (\ref{eq:estimate intrinsic metric non-concentric}), there exists
$C_{\nu}>0$ such that for every $y,\, y^{\prime}\in\overline{S\left(x,s\right)}$,
\[
d\left(y,t;\, y^{\prime},t\right)\leq C_{\nu}\cdot t^{1-\nu/2}\left(\frac{\left|y-y^{\prime}\right|^{1/4}}{t^{1/4}}\wedge1\right).
\]
First we assume that $2\sqrt{\nu}s\leq t$. For any $\epsilon>0$,
$d\left(y,t;\, y^{\prime},t\right)\leq\epsilon$ whenever $\left|y-y^{\prime}\right|\leq C_{\nu}^{-1}\cdot\epsilon^{4}t^{2\nu-3}$.
Therefore, for a possibly larger $C_{\nu}$, 
\[
N\left(\epsilon,\overline{S\left(x,s\right)}\times\left\{ t\right\} \right)\leq C_{\nu}\left(s\epsilon^{-4}t^{3-2\nu}\right)^{\nu}.
\]
Besides, (\ref{eq:estimate intrinsic metric non-concentric}) implies
that 
\[
\mbox{diam}_{d}\left(\overline{S\left(x,s\right)}\times\left\{ t\right\} \right)\leq C_{\nu}\cdot s^{1/4}t^{(3-2\nu)/4}.
\]
By the standard results on entropy (\cite{AT}, Theorem 1.3.3), there
exists a universal constant $K>0$ (later $K$ will be absorbed by
$C_{\nu}$) such that 
\[
\begin{split}\mathbb{E}^{\mathcal{W}}\left[\sup_{y\in\overline{S\left(x,s\right)}}\left|\bar{\theta}_{t}\left(y\right)\right|\right] & \leq K\int_{0}^{\mbox{diam}_{d}\left(\overline{S\left(x,s\right)}\times\left\{ t\right\} \right)/2}\,\sqrt{\ln N\left(\epsilon,\overline{S\left(x,s\right)}\times\left\{ t\right\} \right)}d\epsilon\\
 & \leq4K\nu\int_{0}^{C_{\nu}\cdot s^{1/4}t^{(3-2\nu)/4}}\sqrt{\ln\left(C_{\nu}\cdot s^{1/4}t^{(3-2\nu)/4}\epsilon^{-1}\right)}d\epsilon.\\
 & \leq C_{\nu}\cdot s^{1/4}t^{(3-2\nu)/4}\int_{0}^{\infty}e^{-u^{2}}u^{2}du,
\end{split}
\]
which leads to (\ref{eq:expectation of sup spatial}). 

Next, if $2\sqrt{\nu}s>t$, then $\mbox{diam}_{d}\left(\overline{S\left(x,s\right)}\times\left\{ t\right\} \right)\leq C_{\nu}\cdot t^{1-\nu/2}$.
Following exactly the same arguments as earlier, we arrive at 
\[
\begin{split}\mathbb{E}^{\mathcal{W}}\left[\sup_{y\in\overline{S\left(x,s\right)}}\left|\bar{\theta}_{t}\left(y\right)\right|\right] & \leq C_{\nu}\cdot s^{1/4}t^{(3-2\nu)/4}\int_{\sqrt{\ln\left(C_{\nu}\cdot s^{1/4}t^{-1/4}\right)}}^{\infty}\, e^{-u^{2}}u^{2}du.\end{split}
\]
Combining this with the fact that 
\[
\int_{a}^{\infty}e^{-u^{2}}u^{2}du=\mathcal{O}\left(ae^{-a^{2}}\right)\mbox{ for sufficiently large }a>0,
\]
we arrive at the desired conclusion.
\end{proof}

\subsection{When $p\geq2$.}

As shown in Lemma \ref{lem:covariance C^1(t,s)}, the concentric spherical
averages of the dim-$\nu$ order-1 GFF is a backward Markov Gaussian
process, which enables the renormalization that transforms it into
a time-changed Brownian motion. However, when $\left(I-\Delta\right)$
is replaced by $\left(I-\Delta\right)^{p}$ for $p\geq2$, spherical
averages of the corresponding GFF no longer possess such properties.
In particular, for the dim-$\nu$ order-$p$ GFF with $p\geq2$, for
any fixed $x\in\mathbb{R}^{\nu}$, the concentric spherical average
process $\left\{ \mathcal{I}\left(h_{\sigma_{x,t}^{ave}}\right):\, t\in(0,1]\right\} $
fails to be backward Markovian. Nonetheless, it is still possible
to explicitly compute the covariance of this process, the result of
which shows that, although not being an exact one, the process is
``close'' to becoming a Markov process. To make this rigorous, we
adopt the same method as the one presented in \cite{CJ}. For simplicity,
we only outline the idea here and refer to \cite{CJ} for more details. 

The derivations of the covariance of the spherical averages, as shown
in $\mathsection3.1$ and $\mathsection3.2$, can be generalized to
the operator $m^{2}-\Delta$ for any $m>0$. To be specific, if the
operator $I-\Delta$ is replaced by $m^{2}-\Delta$ in constructing
the dim-$\nu$ order-1 GFF, then for every $x,y\in\mathbb{R}^{\nu}$
and every $t,s\in(0,1]$,
\[
\begin{split} & \mathbb{E}^{\mathcal{W}}\left[\mathcal{I}\left(h_{\sigma_{x,t}^{ave}}\right)\mathcal{I}\left(h_{\sigma_{y,s}^{ave}}\right)\right]\\
 & \qquad\qquad=\frac{\left(2\pi\right)^{\nu/2}}{\alpha_{\nu}^{2}\left(ts\left|x-y\right|\right)^{\frac{\nu-2}{2}}}\int_{0}^{\infty}\frac{\tau^{2-\frac{\nu}{2}}J_{\frac{\nu-2}{2}}\left(t\tau\right)J_{\frac{\nu-2}{2}}\left(s\tau\right)J_{\frac{\nu-2}{2}}\left(\left|x-y\right|\tau\right)}{m^{2}+\tau^{2}}d\tau.
\end{split}
\]
Comparing this expression with the general formula (\ref{eq:covariance for (1-Delta)^s-1}),
one can easily verify that, for the dim-$\nu$ order-$p$ GFF, $\mathbb{E}^{\mathcal{W}^{p}}\left[\mathcal{I}\left(h_{\sigma_{x,t}^{ave}}\right)\mathcal{I}\left(h_{\sigma_{y,s}^{ave}}\right)\right]$
is equal to
\[
\begin{split}\frac{\left(2\pi\right)^{\nu/2}/(p-1)!}{\alpha_{\nu}^{2}\left(ts\left|x-y\right|\right)^{\frac{\nu-2}{2}}}\left(\frac{-1}{2m}\frac{d}{dm}\right)_{m=1}^{p-1}\left[\int_{0}^{\infty}\frac{\tau^{2-\frac{\nu}{2}}J_{\frac{\nu-2}{2}}\left(t\tau\right)J_{\frac{\nu-2}{2}}\left(s\tau\right)J_{\frac{\nu-2}{2}}\left(\left|x-y\right|\tau\right)}{m^{2}+\tau^{2}}d\tau\right].\end{split}
\]
In particular, when $x=y$ and $0<s\leq t\leq1$,
\[
\begin{split} & \mathbb{E}^{\mathcal{W}^{p}}\left[\mathcal{I}\left(h_{\sigma_{x,t}^{ave}}\right)\mathcal{I}\left(h_{\sigma_{x,s}^{ave}}\right)\right]\\
 & \qquad\qquad=\frac{1}{\alpha_{\nu}\left(ts\right)^{\frac{\nu-2}{2}}\left(p-1\right)!}\left(\frac{-1}{2m}\frac{d}{dm}\right)_{m=1}^{p-1}\left[K_{\frac{\nu-2}{2}}\left(mt\right)I_{\frac{\nu-2}{2}}\left(ms\right)\right],
\end{split}
\]
the RHS of which obviously takes the form of 
\begin{equation}
\sum_{k=1}^{p}a_{k}\left(t\right)b_{k}\left(s\right)\label{eq:cov for order-p GFF}
\end{equation}
where functions $a_{k}$ only depend on $t$ and functions $b_{k}$
only depend on $s$ for each $k=1,...,p$. A covariance of the form
of (\ref{eq:cov for order-p GFF}) does indicate that the Gaussian
process $\left\{ \mathcal{I}\left(h_{\sigma_{x,t}^{ave}}\right):t\in(0,1]\right\} $
is not backward Markovian. Heuristically speaking, at any given radius,
the spherical average alone ``provides'' too little information
for one to predict how the process will evolve for smaller radii.
To restore the Markov property, we need to ``collect'' more information%
\footnote{This idea was originally proposed by D. Stroock during a discussion
with the author.%
} about the GFF over each sphere. 

To this end, recall the remark at the end of $\mathsection2$ that
the higher the order of the operator is, the more regular the corresponding
GFF becomes. In particular, for $p\ge2$, the $l-$th derivative of
the spherical average measure in radius, i.e., $\left(d/dt\right)^{l}\sigma_{x,t}^{ave}$
in the sense of tempered distribution, also gives rise to a Paley-Wiener
integral $\mathcal{I}\left(h_{\left(d/dt\right)^{l}\sigma_{x,t}^{ave}}\right)$
for $l=1,\cdots,p-1$. It turns out that, with $x\in\mathbb{R}^{\nu}$
fixed, if $\mathbf{V}_{x,t}$, for $t\in(0,1]$, is the $\mathbb{R}^{p}-$valued
random variable on $\Theta^{p}$ given by 
\[
\mathbf{V}_{x,t}:=\left(\mathcal{I}\left(h_{\sigma_{x,t}^{ave}}\right),\mathcal{I}\left(h_{\left(d/dt\right)\sigma_{x,t}^{ave}}\right),\ldots,\mathcal{I}\left(h_{\left(d/dt\right)^{p-1}\sigma_{x,t}^{ave}}\right)\right),
\]
then the process $\left\{ \mathbf{V}_{x,t}:\, t\in(0,1]\right\} $
is a $\mathbb{R}^{p}-$valued Gaussian (backward) Markov process,
and for $0<s\leq t\leq1$, 
\begin{equation}
\mathbb{E}^{\mathcal{W}^{p}}\left[\left(\mathbf{V}_{x,t}\right)^{\top}\left(\mathbf{V}_{x,s}\right)\right]=\mathbf{A}\left(t\right)\cdot\mathbf{B}\left(s\right),\label{eq:matrix covariance}
\end{equation}
where ``$\cdot$'' here refers to matrix multiplication, $\mathbf{A}\left(t\right)$
and $\mathbf{B}\left(s\right)$ are two $p\times p$ matrices depending
only on $t$ and, respectively, only on $s$, and for $1\leq i,j\leq p$,
\[
\left(\mathbf{A}\left(t\right)\right)_{ij}=\left(\frac{d}{dt}\right)^{i-1}a_{j}\left(t\right)\mbox{ and }\left(\mathbf{B}\left(s\right)\right)_{ij}=\left(\frac{d}{ds}\right)^{j-1}b_{i}\left(s\right),
\]
where $a_{j}$'s and $b_{i}$'s are as in (\ref{eq:cov for order-p GFF}).
In other words, when collecting simultaneously the spherical average
and its first $(p-1)$st order derivatives, the Markov property is
restored by this vector-valued process. Furthermore, the matrix $\mathbf{B}\left(s\right)$
is non-degenerate when $s$ is sufficiently small, so (\ref{eq:matrix covariance})
also leads to a renormalization which is $\mathbf{U}_{x,t}:=\mathbf{V}_{x,t}\cdot\mathbf{B}^{-1}\left(t\right)$.
It follows from (\ref{eq:matrix covariance}) that, for $0<s\leq t\leq1$,
\[
\mathbb{E}^{\mathcal{W}^{p}}\left[\left(\mathbf{U}_{x,t}\right)^{\top}\left(\mathbf{U}_{x,s}\right)\right]=\mathbf{B}^{-1}\left(t\right)\cdot\mathbf{A}\left(t\right).
\]
The renormalized process $\left\{ \mathbf{U}_{x,t}:\, t\in(0,1]\right\} $
has independent increment (in the direction of $t$ decreasing). Moreover,
it is possible to find a constant vector $\xi\in\mathbb{R}^{p}$ such
that, as $t\searrow0$, 
\[
\left(\sigma_{x,t}^{ave},\,(d/dt)\sigma_{x,t}^{ave},\cdots,\,(d/dt)^{p-1}\sigma_{x,t}^{ave}\right)\cdot\mathbf{B}^{-1}\left(t\right)\cdot\xi^{\top}\rightarrow\delta_{x}
\]
in the sense of tempered distribution; this is because the coefficient
of $(d/dt)^{l}\sigma_{x,t}^{\mbox{ave}}$, as a function of $t$,
decays sufficiently fast as $t\searrow0$ for each $l=1,\cdots,p-1$.
Therefore, $\bar{\theta}_{t}\left(x\right):=\mathbf{U}_{x,t}\left(\theta\right)\cdot\xi^{\top}$
still is a legitimate approximation of ``$\theta\left(x\right)$''
when $t$ is small. In other words, although the derivatives of the
spherical averages are introduced to recover the Markov property,
these derivatives do not affect the approximation of point-wise values
of the GFF. Moreover, the two-parameter family $\left\{ \bar{\theta}_{t}\left(x\right):\, x\in\mathbb{R}^{\nu},\, t\in(0,1]\right\} $
possesses the same properties as those shown in $\mathsection3.1$
and $\mathsection3.2$. 

Governed by the Green's function of $\left(I-\Delta\right)^{p}$ on
$\mathbb{R}^{\nu}$, the dim-$\nu$ order-$p$ GFF is polynomially
correlated with the degree of the polynomial being $\nu-2p$. In fact,
later discussions in this article, i.e., the study of thick point,
only requires the existence of an approximation $\bar{\theta}_{t}\left(x\right)$
such as the one obtained above. Therefore, it is sufficient to assume
$p=1$ and investigate the thick point problem for the dim-$\nu$
order-1 GFF with arbitrary $\nu>2$.

\section{Thick Points of Gaussian Free Fields}

To study the thick points of the dim-$\nu$ order-$1$ GFF $\left(H,\Theta,\mathcal{W}\right)$,
the first problem we face is to determine a proper definition for
the notion of ``thick point''. On one hand, inspired by (\ref{eq:2D thick point})
the thick point definition of log-correlated GFFs, we want to investigate
the points $x\in\overline{S\left(O,1\right)}$ where the rate of $\bar{\theta}_{t}\left(x\right)$
``blowing up'' as $t\searrow0$ is comparable with certain function
in $t$ that is singular at $t=0$. On the other hand, compared with
the log-correlated GFFs, a polynomial-correlated GFF has the properties
that, firstly, the point-wise distribution has a larger variance which
makes it harder to achieve an ``unusually'' large value required
by a limit such as the one in (\ref{eq:2D thick point}); secondly,
the near-neighbor correlation is stronger, which makes thick points,
defined in any reasonable sense, tend to stay close to each other,
and hence as a whole the set of thick points looks more sparse. Taking
into account of these considerations, we adopt a thick point definition
that is different from (\ref{eq:2D thick point}) but proven to be
more suitable for polynomial-correlated GFFs.
\begin{defn}
Let $\gamma\geq0$. For each $\theta\in\Theta$, $x\in\overline{S\left(O,1\right)}$
is a $\gamma-$\emph{thick point }of $\theta$ if
\begin{equation}
\limsup_{t\searrow0}\,\frac{\bar{\theta}_{t}\left(x\right)}{\sqrt{-G\left(t\right)\ln t}}\geq\sqrt{2\nu\gamma}\label{eq:def of thick point}
\end{equation}
where $G\left(t\right)=\mathbb{E}^{\mathcal{W}}\left[\left(\bar{\theta}_{t}\left(x\right)\right)^{2}\right]$. 
\end{defn}
We denote by $T_{\theta}^{\gamma}$ the set of all the $\gamma-$thick
points of $\theta$. Since $\bar{\theta}_{t}\left(x\right)$ is assumed
to be continuous in $\left(x,t\right)\in\overline{S\left(O,1\right)}\times(0,1]$,
$T_{\theta}^{\gamma}$ is a measurable subset of $\overline{S\left(O,1\right)}$.
Moreover, viewing from the perspective of (\ref{eq:LIL for concentric}),
if $\gamma>0$, (\ref{eq:def of thick point}) requires $\bar{\theta}_{t}\left(x\right)$
to grow, as $t\searrow0$, no slower than a unusually large function
in $t$, at least along a sequence in $t$. In this sense, the value
of $\theta$ is unusually large at a $\gamma-$thick point so long
as $\gamma>0$. Compared with (\ref{eq:2D thick point}) , the requirement
in (\ref{eq:def of thick point}) is easier to achieve, which contributes
positively to $T_{\theta}^{\gamma}$ having ``detectable'' mass.
In fact, such a deviation from (\ref{eq:2D thick point}) (i.e., replacing
in the definition ``$\lim_{t\searrow0}$'' by ``$\limsup_{t\searrow0}$''
and ``$=$'' by ``$\geq$'') is necessary, as we will see later
in $\mathsection4.2$.

Our main goal is to determine the Hausdorff dimension of $T_{\theta}^{\gamma}$,
denoted by $\dim_{\mathcal{H}}\left(T_{\theta}^{\gamma}\right)$.
We state the main result below.
\begin{thm}
\label{thm:hausdorff dimension of thick point set}For $\gamma\in\left[0,1\right]$,
\[
\dim_{\mathcal{H}}\left(T_{\theta}^{\gamma}\right)=\nu\left(1-\gamma\right)\mbox{ a.s.}.
\]
Moreover, for $\mathcal{W}-$a.e. $\theta\in\Theta$, $x\in T_{\theta}^{0}$
for Lebesgue-a.e. $x\in\overline{S\left(O,1\right)}$. 

On the other hand, for $\gamma>1$, $T_{\theta}^{\gamma}=\emptyset$
a.s..
\end{thm}
The theorem is proven by establishing the upper bound and the lower
bound separately. More specifically, we prove in $\mathsection4.1$
\begin{equation}
the\; upper\; bound:\quad\dim_{\mathcal{H}}\left(T_{\theta}^{\gamma}\right)\leq\nu\left(1-\gamma\right)\mbox{ a.s.}.\label{eq:upper bound}
\end{equation}
Then we devote the entire $\mathsection5$ to proving
\begin{equation}
the\; lower\; bound:\quad\dim_{\mathcal{H}}\left(T_{\theta}^{\gamma}\right)\geq\nu\left(1-\gamma\right)\mbox{ a.s.}.\label{eq:lower bound}
\end{equation}
As mentioned earlier, the polynomially singular covariance of the
GFF makes thick points rare and hence hard to detect, as a consequence
of which, the upper bound on the Hausdorff dimension of $T_{\theta}^{\gamma}$
is readily obtained, but deriving the lower bound is more complicated.

\subsection{Proof of the Upper Bound }

The derivation of the upper bound (\ref{eq:upper bound}) follows
an adaptation of the procedure used in \cite{HMP}. To simplify the
notation, we write 
\[
D\left(t\right):=\sqrt{-G\left(t\right)\ln t}\;\mbox{ for }t\in(0,1].
\]

\begin{lem}
\label{lem:extreme value of two parameter family}There exists a constant
$C_{\nu}>0$ such that for every $x\in\overline{S\left(O,1\right)}$
and $n\geq1$, 
\begin{equation}
\mathbb{E}^{\mathcal{W}}\left[\sup_{\left(y,t\right)\in\overline{S\left(x,2^{-n}\right)}\times\left[2^{-n},\,2^{-n+1}\right]}\,\frac{\bar{\theta}_{t}\left(y\right)}{D\left(t\right)}\right]\leq C_{\nu}\cdot\frac{1}{\sqrt{n}}.\label{eq:expectation of max of two parameter family}
\end{equation}
\end{lem}
\begin{proof}
Similarly as in Lemma \ref{lem:expectation of max non-concentric},
we will prove the desired result by the metric entropy method. Let
$n\geq1$ be fixed. For every $\epsilon>0$, set
\[
\tau_{n,\epsilon}:=\frac{1}{2}\left[\left(\frac{\epsilon^{2}}{9}\cdot C_{\nu}^{-1}\cdot2^{-\left(n+1\right)\left(\nu-3/2\right)}\right)^{2}\wedge2^{-n-1}\right]
\]
where $C_{\nu}$, for the moment, is the same constant as in (\ref{eq:estimate intrinsic metric non-concentric}).
Let 
\[
\left\{ B\left(y_{l},\tau_{n,\epsilon}\right):\, l=1,\cdots,L_{\epsilon}\right\} 
\]
be a finite covering of $\overline{S\left(x,2^{-n}\right)}$ where
$y_{l}\in S\left(x,2^{-n}\right)$ and $L_{\epsilon}$ is the smallest
number of balls $B\left(y_{l},\tau_{n,\epsilon}\right)$ needed to
cover $\overline{S\left(x,2^{-n}\right)}$ and hence $L_{\epsilon}=\mathcal{O}\left(2^{-n\nu}/\tau_{n,\epsilon}^{\nu}\right)$.
By (\ref{eq:estimate intrinsic metric non-concentric}), the choice
of $\tau_{n,\epsilon}$ is such that the diameter of each ball $B\left(y_{l},\tau_{n,\epsilon}\right)$
under the metric $d\left(\cdot,2^{-n-1};\,*,2^{-n-1}\right)$ is no
greater than $\epsilon/3$. In fact, for any $t\geq2^{-n-1}$, (\ref{eq:estimate intrinsic metric non-concentric})
implies that, if $\left|y-y^{\prime}\right|\leq2\tau_{n,\epsilon}$,
then 
\[
d^{2}\left(y,t;\, y^{\prime},t\right)\leq C_{\nu}t^{3/2-\nu}\sqrt{2\tau_{n,\epsilon}}\leq\epsilon^{2}/9.
\]
Next, take $\tau_{0}:=2^{-n-1}$ and define $\tau_{m}$ inductively
such that 
\[
G\left(\tau_{m-1}\right)-G\left(\tau_{m}\right)=\epsilon^{2}/9
\]
for $m=1,\cdots,M_{\epsilon}$, where $M_{\epsilon}$ is the smallest
integer such that $\tau_{M_{\epsilon}}\geq2^{-n+2}$ and hence 
\[
M_{\epsilon}=\mathcal{O}\left(G\left(2^{-n}\right)\right)/\epsilon^{2}.
\]
Consider the covering of $\overline{S\left(x,2^{-n}\right)}\times\left[2^{-n},2^{-n+1}\right]$
that consists of the cylinders 
\[
\left\{ B\left(y_{l},\tau_{n,\epsilon}\right)\times(\tau_{m-1},\tau_{m}):\, l=1,\cdots,L_{\epsilon},\, m=1,\cdots,M_{\epsilon}\right\} .
\]
Any pair of points $\left(\left(y,t\right),\,\left(w,s\right)\right)$
that lies in one of the cylinders above, e.g., $B\left(y_{l},\tau_{n,\epsilon}\right)\times(\tau_{m-1},\tau_{m})$,
satisfies that 
\[
\begin{split}d\left(y,t;\, w,s\right) & \leq d\left(y,t;\, y,\tau_{m}\right)+d\left(y,\tau_{m};\, w,\tau_{m}\right)+d\left(w,\tau_{m};\, w,s\right)\\
 & \leq\epsilon/3+\epsilon/3+\epsilon/3=\epsilon.
\end{split}
\]
This implies that 
\[
N\left(\epsilon,\overline{S\left(x,2^{-n}\right)}\times\left[2^{-n},2^{-n+1}\right]\right)\leq L_{\epsilon}\cdot M_{\epsilon}
\]
where $N$ is the entropy function defined before Lemma \ref{lem:expectation of max non-concentric}.
Moreover, the diameter of $\overline{S\left(x,2^{-n}\right)}\times\left[2^{-n},2^{-n+1}\right]$
under the metric $d$ is bounded by $2\sqrt{G\left(2^{-n}\right)}$.
Therefore, there is a universal constant $K>0$ (later $K$ is absorbed
into $C_{\nu}$) such that
\[
\begin{split}\mathbb{E}^{\mathcal{W}}\left[\sup_{\left(y,t\right)\in\overline{S\left(x,2^{-n}\right)}\times\left[2^{-n},\,2^{-n+1}\right]}\,\bar{\theta}_{t}\left(y\right)\right] & \leq K\int_{0}^{\sqrt{G\left(2^{-n}\right)}}\sqrt{\ln\left(L_{\epsilon}\cdot M_{\epsilon}\right)}d\epsilon\\
 & \leq K\int_{0}^{\sqrt{G\left(2^{-n}\right)}}\left(\sqrt{\ln L_{\epsilon}}+\sqrt{\ln M_{\epsilon}}\right)d\epsilon\\
 & \leq C_{\nu}\sqrt{G\left(2^{-n}\right)}.
\end{split}
\]
(\ref{eq:expectation of max of two parameter family}) follows from
dividing both sides of the inequality above by $D\left(2^{-n+1}\right)$.
\end{proof}
Now we can get to the proof of the upper bound (\ref{eq:upper bound}).\\

\noindent \emph{Proof of the upper bound:} When $\gamma=0$, (\ref{eq:upper bound})
is trivially satisfied. Without loss of generality, we assume that
$\gamma\in(0,1]$ for the rest of the proof. For each $n\geq0$, consider
a finite lattice partition of $\overline{S\left(O,1\right)}$ with
cell size $2\cdot2^{-n}$ (i.e., the length, under the Euclidean metric,
of each side of the cell is $2\cdot2^{-n}$). Let $\left\{ x_{j}^{\left(n\right)}:\, j=1,\cdots,J_{n}\right\} $
be the collection of the lattice cell centers where $J_{n}=2^{\nu n}$
is the total number of the cells. Let $\gamma^{\prime\prime},\gamma^{\prime}$
be two numbers such that $0<\gamma^{\prime\prime}<\gamma^{\prime}<\gamma$
and $\gamma^{\prime}$ and $\gamma^{\prime\prime}$ can be arbitrarily
close to $\gamma$. Consider the subset of the indices 
\[
I_{n}:=\left\{ j:\;1\leq j\leq J_{n}\mbox{ s.t. }\sup_{\left(y,r\right)\in\overline{S\left(x_{j}^{\left(n\right)},2^{-n}\right)}\times\left[2^{-n},2^{-n+1}\right]}\,\frac{\bar{\theta}_{r}\left(y\right)}{D\left(r\right)}>\sqrt{2\nu\gamma^{\prime}}\right\} .
\]
Combining (\ref{eq:expectation of max of two parameter family}) and
the Borell-TIS inequality (\cite{AT} $\mathsection2.1$ and the references
therein), we have that, for every $j=1,\cdots,J_{n}$,
\[
\begin{split}\mathbb{\mathcal{W}}\left(j\in I_{n}\right) & =\mathcal{W}\left(\sup_{\left(y,r\right)\in\overline{S\left(x_{j}^{\left(n\right)},2^{-n}\right)}\times\left[2^{-n},2^{-n+1}\right]}\,\frac{\bar{\theta}_{r}\left(y\right)}{D\left(r\right)}>\sqrt{2\nu\gamma^{\prime}}\right)\\
 & \leq\exp\left[-\frac{1}{2}\left(\sqrt{2\nu\gamma^{\prime}}-\frac{C_{\nu}}{\sqrt{n}}\right)^{2}\cdot\ln2\cdot\left(n-1\right)\right]\\
 & \leq\exp\left(-\nu\gamma^{\prime\prime}\cdot\ln2\cdot n\right)\,.
\end{split}
\]
Therefore, 
\[
\mathbb{E}^{\mathcal{W}}\left[\mbox{card}\left(I_{n}\right)\right]=\sum_{j\in J_{n}}\mathcal{W}\left(j\in I_{n}\right)\leq C_{\nu}\cdot2^{\nu\left(1-\gamma^{\prime\prime}\right)n}.
\]

On the other hand, if $y\in T_{\theta}^{\gamma}$, then there exists
a sequence $\left\{ t_{k}:\, k\geq0\right\} $ with $t_{k}\searrow0$
as $k\nearrow\infty$ such that 
\[
\frac{\bar{\theta}_{t_{k}}\left(y\right)}{D\left(t_{k}\right)}>\sqrt{2\nu\gamma^{\prime}}\mbox{ for all }k\geq0.
\]
For every $k$, let $n(k)$ be the unique positive integer such that
\[
2^{-n(k)}\leq t_{k}<2^{-n(k)+1}.
\]
If $x_{j}^{\left(n(k)\right)}$ is the cell center (at $n(k)-$th
level) such that $\left|y-x_{j}^{\left(n(k)\right)}\right|\leq\sqrt{\nu}2^{-n(k)}$,
then clearly $j\in I_{n(k)}$. Therefore, 
\[
T_{\theta}^{\gamma}\subseteq\bigcap_{m\geq1}\,\bigcup_{n\geq m}\,\bigcup_{j\in I_{n}}\,\overline{S\left(x_{j}^{(n)},2^{-n}\right)}.
\]
Moreover, for each $m\geq1$, $\left\{ \overline{S\left(x_{j}^{\left(n\right)},2^{-n}\right)}:j\in I_{n},\, n\geq m\right\} $
forms a covering of $T_{\theta}^{\gamma}$, and the diameter (under
the Euclidean metric) of $\overline{S\left(x_{j}^{\left(n\right)},2^{-n}\right)}$
is $\sqrt{\nu}2^{-n+1}$. Thus, if $\mathcal{H}^{\alpha}$ is the
Hausdorff-$\alpha$ measure for $\alpha>0$, then
\[
\begin{split}\mathcal{H}^{\alpha}\left(T_{\theta}^{\gamma}\right) & \leq\liminf_{m\rightarrow\infty}\,\sum_{n\geq m}\,\sum_{j\in I_{n}}\,\left(2\sqrt{\nu}\cdot2^{-n}\right)^{\alpha}\\
 & =C_{\nu}\cdot\liminf_{m\rightarrow\infty}\,\sum_{n\geq m}\,\mbox{card}\left(I_{n}\right)2^{-n\alpha}.
\end{split}
\]
It follows from Fatou's lemma that
\[
\begin{split}\mathbb{E}^{\mathcal{W}}\left[\mathcal{H}^{\alpha}\left(T_{\theta}^{\gamma}\right)\right] & \leq C_{\nu}\cdot\liminf_{m\rightarrow\infty}\,\sum_{n\geq m}\,\mathbb{E}^{\mathcal{W}}\left[\mbox{card}\left(I_{n}\right)\right]\cdot2^{-n\alpha}\\
 & \leq C_{\nu}\cdot\lim_{m\rightarrow\infty}\,\sum_{n\geq m}\,2^{\left[\nu\left(1-\gamma^{\prime\prime}\right)-\alpha\right]n}\,.
\end{split}
\]
Clearly, for any $\alpha>\nu\left(1-\gamma^{\prime\prime}\right)$,
$\mathcal{H}^{\alpha}\left(T_{\theta}^{\gamma}\right)=0$ a.s. and
hence $\dim_{\mathcal{H}}\left(T_{\theta}^{\gamma}\right)\leq\alpha$
a.s.. Since $\gamma^{\prime\prime}$ is arbitrarily close to $\gamma$,
we conclude that 
\[
\dim_{\mathcal{H}}\left(T_{\theta}^{\gamma}\right)\leq\nu\left(1-\gamma\right)\mbox{ a.s..}
\]

We have completed the proof of the upper bound (\ref{eq:upper bound}).
In addition, by the same argument as above, if $\gamma>1$, we can
choose $\gamma^{\prime\prime}$ to be greater than $1$, in which
case 
\[
\sum_{n\geq m}\mathbb{E}^{\mathcal{W}}\left[\mbox{card}\left(I_{n}\right)\right]\leq\sum_{n\geq m}\,2^{\nu\left(1-\gamma^{\prime\prime}\right)n}\rightarrow0\mbox{ as }m\rightarrow\infty.
\]
This observation immediately implies the last assertion in Theorem
\ref{thm:hausdorff dimension of thick point set}, i.e., if $\gamma>1$,
then $T_{\theta}^{\gamma}=\emptyset$ a.s..

\subsection{Perfect $\gamma-$thick point.}

In this subsection we explain why the definition (\ref{eq:def of thick point})
is more proper for the study of thick points for polynomial-correlated
GFFs. Simply speaking, the straightforward analogue of (\ref{eq:2D thick point}),
the thick point definition for log-correlated GFFs, imposes too strong
a condition to fulfill in the case of polynomial-correlated GFFs.
To make this precise, we first define a more strict analogue of (\ref{eq:2D thick point}).
\begin{defn}
Let $\gamma\geq0$. For each $\theta\in\Theta$, $x\in\overline{S\left(O,1\right)}$
is called a \emph{perfect} $\gamma-$\emph{thick point }of $\theta$
if
\[
\lim_{t\searrow0}\,\frac{\bar{\theta}_{t}\left(x\right)}{\sqrt{-G\left(t\right)\ln t}}=\sqrt{2\nu\gamma}.
\]

\end{defn}
Again, if $PT_{\theta}^{\gamma}$ is the set that contains all the
perfect $\gamma-$thick points of $\theta$, then $PT_{\theta}^{\gamma}$
is a measurable subset of $\overline{S\left(O,1\right)}$. To study
$PT_{\theta}^{\gamma}$, we follow a similar strategy as the one used
to establish the upper bound in $\mathsection4.1$. For each $n\geq0$,
set $s_{n}:=2^{-n^{2}}$. For $n\geq1$, consider the two-parameter
Gaussian family

\[
\mathcal{A}_{n}:=\left\{ \bar{\theta}_{t}\left(x\right):\, x\in\overline{S\left(O,1\right)},\, t\in[s_{n},\, s_{n-1}]\right\} .
\]
Let $\omega_{n}$ be the modulus of continuity of $\mathcal{A}_{n}$
under the intrinsic metric $d$, i.e., for every $\delta>0$, 
\[
\omega_{n}\left(\delta\right):=\sup\left\{ \left|\bar{\theta}_{t}\left(x\right)-\bar{\theta}_{s}\left(y\right)\right|:\, d\left(x,t;y,s\right)\leq\delta,\, x,y\in\overline{S\left(O,1\right)},\, t,s\in[s_{n},\, s_{n-1}]\right\} .
\]

\begin{lem}
There exists a constant $C_{\nu}>0$ such that for every $n\geq1$
and every $0<\delta\ll\sqrt{G\left(s_{n}\right)}$, 
\begin{equation}
\mathbb{E}^{\mathcal{W}}\left[\omega_{n}\left(\delta\right)\right]\leq C_{\nu}\cdot\delta\sqrt{\ln\left(s_{n}^{\left(3-2\nu\right)/4}/\delta\right)}.\label{eq:expectation of modulus of continuity}
\end{equation}
Moreover, if 
\[
\mathcal{B}_{n}:=\left\{ \left(x,\, y\right)\in\left(\overline{S\left(O,1\right)}\right)^{2}:\,\left|x-y\right|\leq\sqrt{\nu}\cdot s_{2n}\right\} ,
\]
then 
\begin{equation}
\mathcal{W}\left(\sup_{\left(x,\, y\right)\in\mathcal{B}_{n},\, t\in[s_{n},\, s_{n-1}]}\left|\frac{\bar{\theta}_{t}\left(x\right)}{D\left(t\right)}-\frac{\bar{\theta}_{t}\left(y\right)}{D\left(t\right)}\right|>2^{-3n^{2}/16},\;\mbox{i.o.}\right)=0.\label{eq:almost sure modulus}
\end{equation}
\end{lem}
\begin{proof}
For every $\epsilon>0$, let $N\left(\epsilon,\mathcal{A}_{n}\right)$
be the entropy function, as introduced before Lemma \ref{lem:expectation of max non-concentric}.
Then, it follows from a similar argument as the one used in the proof
of Lemma \ref{lem:extreme value of two parameter family} that $N\left(\epsilon,\mathcal{A}_{n}\right)\leq L_{\epsilon}\cdot M_{\epsilon}$
where $L_{\epsilon}=\mathcal{O}\left(s_{n}^{\nu\left(3-2\nu\right)}\right)/\epsilon^{4\nu}$
is the entropy for the set $\overline{S\left(O,1\right)}$ under the
metric $d\left(\cdot,s_{n};\,*,s_{n}\right)$, and $M_{\epsilon}=\mathcal{O}\left(G\left(s_{n}\right)\right)/\epsilon^{2}$
is the entropy for the interval $[s_{n},s_{n-1}]$ corresponding to
the concentric process $\left\{ \bar{\theta}_{t}\left(x\right):t\in[s_{n},s_{n-1}]\right\} $
for any fixed $x$. Hence (\cite{AT}, Corollary 1.3.4), there exists
a universal constant $K>0$ such that for every $n\geq1$, 
\[
\begin{split}\mathbb{E}^{\mathcal{W}}\left[\omega_{n}\left(\delta\right)\right] & \leq K\int_{0}^{\delta}\sqrt{\ln N\left(\epsilon,\mathcal{A}_{n}\right)}d\epsilon\leq K\int_{0}^{\delta}\left(\sqrt{\ln L_{\epsilon}}+\sqrt{\ln M_{\epsilon}}\right)d\epsilon.\end{split}
\]
Similarly as the integrals we evaluated when proving (\ref{eq:expectation of sup spatial}),
we have that 
\[
\int_{0}^{\delta}\sqrt{\ln L_{\epsilon}}d\epsilon\leq C_{\nu}\cdot\delta\sqrt{\ln\left(s_{n}^{\left(3-2\nu\right)/4}/\delta\right)}
\]
and 
\[
\int_{0}^{\delta}\sqrt{\ln M_{\epsilon}}d\epsilon\leq C_{\nu}\cdot\delta\sqrt{\ln\left(s_{n}^{\left(1-\nu/2\right)}/\delta\right)}
\]
for some $C_{\nu}>0$, which lead to the first assertion.

To prove the second assertion, notice that by (\ref{eq:estimate intrinsic metric non-concentric}),
if $\left|x-y\right|\leq\sqrt{\nu}\cdot s_{2n}$, then for any $t\in\left[s_{n},\, s_{n-1}\right]$,
\[
d\left(x,t;\, y,t\right)\leq C_{\nu}\cdot t^{\left(3-2\nu\right)/4}\cdot s_{2n}^{1/4}\leq C_{\nu}\cdot2^{n^{2}\left(2\nu-7\right)/4}.
\]
Therefore, by (\ref{eq:expectation of modulus of continuity}), 
\[
\begin{split}\mathbb{E}^{\mathcal{W}}\left[\sup_{\left(x,\, y\right)\in\mathcal{B}_{n},\, t\in[s_{n},\, s_{n-1}]}\left|\bar{\theta}_{t}\left(x\right)-\bar{\theta}_{t}\left(y\right)\right|\right] & \leq\mathbb{E}^{\mathcal{W}}\left[\omega_{n}\left(C_{\nu}\cdot2^{n^{2}\left(\nu/2-7/4\right)}\right)\right]\\
 & \leq C_{\nu}\cdot2^{n^{2}\left(\nu/2-7/4\right)}\cdot n.
\end{split}
\]
The desired conclusion follows from dividing both sides of the inequality
above by $D\left(s_{n-1}\right)$ and applying the Borel-Cantelli
lemma.
\end{proof}
Now we are ready to establish the main result of this subsection,
which is, if $\gamma>0$, then the perfect $\gamma-$thick point doesn't
exist almost surely. Being ``perfect'' prevents such points from
existing.
\begin{thm}
\label{thm: no perfect thick piont}If $\gamma>0$, then $PT_{\theta}^{\gamma}=\emptyset$
a.s..\end{thm}
\begin{proof}
Based on (\ref{eq:almost sure modulus}), for $\mathcal{W}-$a.e.
$\theta$, there exists $N_{\theta}\in\mathbb{N}$ such that for every
$n\geq N_{\theta}$ and $x,y$ such that $y\in\overline{S\left(x,s_{2n}\right)}$,
\[
\sup_{t\in[s_{n},\, s_{n-1}]}\left|\frac{\bar{\theta}_{t}\left(x\right)}{D\left(t\right)}-\frac{\bar{\theta}_{t}\left(y\right)}{D\left(t\right)}\right|\leq2^{-3n^{2}/16}.
\]
Choose $M>0$ to be a sufficiently large constant. Consider the lattice
partition of $\overline{S\left(O,1\right)}$ with cell size $2^{-k}$,
and let 
\[
\left\{ x_{j}^{(k)}:\, j=1,\cdots,J_{k}\right\} 
\]
be the cell centers. Let $y$ be a perfect $\gamma-$thick point.
For $n$ that is sufficiently large, if $x_{j}^{(4n^{2})}$ is the
cell center such that $y\in\overline{S\left(x_{j}^{(4n^{2})},s_{2n}\right)}$,
then 
\[
\sup_{t\in[s_{n},\, s_{n-1}]}\left|\frac{\bar{\theta}_{t}\left(x_{j}^{(4n^{2})}\right)}{D\left(t\right)}-\sqrt{2\nu\gamma}\right|\leq\frac{1}{M}.
\]
In particular, this means that 
\[
\left|\bar{\theta}_{s_{n-1}}\left(x_{j}^{(4n^{2})}\right)-\sqrt{2\nu\gamma}D\left(s_{n-1}\right)\right|\leq\frac{D\left(s_{n-1}\right)}{M},
\]
and for every $t\in\left[s_{n},\, s_{n-1}\right]$ , 
\[
\left|\bar{\theta}_{t}\left(x_{j}^{(4n^{2})}\right)-\bar{\theta}_{s_{n-1}}\left(x_{j}^{(4n^{2})}\right)-\sqrt{2\nu\gamma}\left[D\left(t\right)-D\left(s_{n-1}\right)\right]\right|\leq\frac{2D\left(t\right)}{M}.
\]

Let $\mathcal{P}_{n}=\mathcal{P}_{n}^{x_{j}^{(4n^{2})}}\subseteq\Theta$
be the collection of all the field elements $\theta$ such that 
\[
\forall t\in\left[s_{n},\, s_{n-1}\right],\left|\bar{\theta}_{t}\left(x_{j}^{(4n^{2})}\right)-\bar{\theta}_{s_{n-1}}\left(x_{j}^{(4n^{2})}\right)-\sqrt{2\nu\gamma}\left[D\left(t\right)-D\left(s_{n-1}\right)\right]\right|\leq\frac{2D\left(t\right)}{M}.
\]
Clearly, $\mathcal{P}_{n}$ is a measurable set and $\mathcal{W}\left(\mathcal{P}_{n}\right)$
doesn't depend on $x_{j}^{(4n^{2})}$. To simplify the notation, we
will write ``$x_{j}^{(4n^{2})}$'' as ``$x$'' throughout this
proof. The idea is to rewrite $\mathcal{P}_{n}$ in terms of a shifted
$\theta$ and apply the Cameron-Martin formula. To this end, we define,
for $t\in(0,1]$,
\[
F\left(t\right):=\frac{D^{\prime}\left(t\right)}{G^{\prime}\left(t\right)}\mbox{ and }f\left(t\right):=F^{\prime}\left(t\right),
\]
and let $\zeta_{x,n}$ be the element in $H^{-1}\left(\mathbb{R}^{\nu}\right)$
such that its corresponding Paley-Wiener integral is 
\[
\mathcal{I}\left(h_{\zeta_{x,n}}\right)\left(\theta\right)=\int_{s_{n}}^{s_{n-1}}\left(\bar{\theta}_{t}\left(x\right)-\bar{\theta}_{s_{n-1}}\left(x\right)\right)f\left(t\right)dt+F\left(s_{n}\right)\left(\bar{\theta}_{s_{n}}\left(x\right)-\bar{\theta}_{s_{n-1}}\left(x\right)\right).
\]
We observe that for every $t\in\left[s_{n},\, s_{n-1}\right]$, 
\[
\begin{split}\left(h_{\bar{\sigma}_{x,t}},\, h_{\zeta_{x,n}}\right)_{H} & =\int_{s_{n}}^{t}\left[G\left(t\right)-G\left(s_{n-1}\right)\right]f\left(s\right)ds+\int_{t}^{s_{n-1}}\left[G\left(s\right)-G\left(s_{n-1}\right)\right]f\left(s\right)ds\\
 & \qquad\qquad\qquad\qquad\qquad\qquad\qquad\qquad\qquad+F\left(s_{n}\right)\left[G\left(t\right)-G\left(s_{n-1}\right)\right]\\
 & =\left[G\left(t\right)-G\left(s_{n-1}\right)\right]F\left(t\right)+\int_{t}^{s_{n-1}}\left[G\left(s\right)-G\left(s_{n-1}\right)\right]f\left(s\right)ds\\
 & =-\int_{t}^{s_{n-1}}G^{\prime}\left(s\right)F\left(s\right)ds=-\int_{t}^{s_{n-1}}D^{\prime}\left(s\right)ds\\
 & =D\left(t\right)-D\left(s_{n-1}\right).
\end{split}
\]
Therefore,
\[
\mathcal{I}\left(h_{\bar{\sigma}_{x,t}}-h_{\bar{\sigma}_{x,s_{n-1}}}\right)\left(\theta-\sqrt{2\nu\gamma}h_{\zeta_{x,n}}\right)=\bar{\theta}_{t}\left(x\right)-\bar{\theta}_{s_{n-1}}\left(x\right)-\sqrt{2\nu\gamma}\left[D\left(t\right)-D\left(s_{n-1}\right)\right].
\]
 In other words, we can view 
\[
\bar{\theta}_{t}\left(x\right)-\bar{\theta}_{s_{n-1}}\left(x\right)-\sqrt{2\nu\gamma}\left[D\left(t\right)-D\left(s_{n-1}\right)\right]
\]
as a Paley-Wiener integral of a translated GFF. Thus, by the Cameron-Martin
formula (\cite{probability}, Theorem 8.2.9), $\mathcal{W}\left(\mathcal{P}_{n}\right)$
is equal to 
\[
\begin{split} & \mathbb{E}^{\mathcal{W}}\left[e^{-\sqrt{2\nu\gamma}\mathcal{I}\left(h_{\zeta_{x,n}}\right)-\nu\gamma\left\Vert h_{\zeta_{x,n}}\right\Vert _{H}^{2}};\,\forall t\in\left[s_{n},\, s_{n-1}\right],\left|\bar{\theta}_{t}\left(x\right)-\bar{\theta}_{s_{n-1}}\left(x\right)\right|\leq\frac{2D\left(t\right)}{M}\right]\\
\le\; & e^{-\nu\gamma\left\Vert h_{\zeta_{x,n}}\right\Vert _{H}^{2}}\cdot\exp\left\{ \frac{2\sqrt{2\nu\gamma}}{M}\left[\int_{s_{n}}^{s_{n-1}}D\left(t\right)f\left(t\right)dt+F\left(s_{n}\right)D\left(s_{n}\right)\right]\right\} \\
=\; & e^{-\nu\gamma\left\Vert h_{\zeta_{x,n}}\right\Vert _{H}^{2}}\cdot\exp\left\{ \frac{2\sqrt{2\nu\gamma}}{M}\left[\int_{s_{n}}^{s_{n-1}}\left(-D^{\prime}\left(t\right)\right)F\left(t\right)dt+F\left(s_{n-1}\right)D\left(s_{n-1}\right)\right]\right\} .
\end{split}
\]
Moreover, we compute 
\[
\begin{split}\left\Vert h_{\zeta_{x,n}}\right\Vert _{H}^{2} & =\int_{s_{n}}^{s_{n-1}}\left[D\left(t\right)-D\left(s_{n-1}\right)\right]f\left(t\right)dt+F\left(s_{n}\right)\left[D\left(s_{n}\right)-D\left(s_{n-1}\right)\right]\\
 & =\int_{s_{n}}^{s_{n-1}}\left(-D^{\prime}\left(t\right)\right)F\left(t\right)dt.
\end{split}
\]
It is easy to verify that, when $n$ is large, 
\[
\int_{s_{n}}^{s_{n-1}}\left(-D^{\prime}\left(t\right)\right)F\left(t\right)dt=\mathcal{O}\left(n^{3}\right)\mbox{ and }F\left(s_{n-1}\right)D\left(s_{n-1}\right)=\mathcal{O}\left(n^{2}\right).
\]
Thus, $\mathcal{W}\left(\mathcal{P}_{n}\right)$ is no greater than
\[
\begin{split} & \exp\left[\left(-\nu\gamma+\frac{2\sqrt{2\nu\gamma}}{M}\right)\mathcal{O}\left(n^{3}\right)+\frac{2\sqrt{2\nu\gamma}}{M}\mathcal{O}\left(n^{2}\right)\right]\leq\exp\left(-n^{3}/M\right)\end{split}
\]
when $M$ is sufficiently large.

To complete the proof, we repeat the arguments that lead to the last
assertion in Theorem \ref{thm:hausdorff dimension of thick point set}.
Because 
\[
\left\{ \theta\in\Theta:\, PT_{\theta}^{\gamma}\neq\emptyset\right\} \subseteq\bigcap_{m\geq1}\,\bigcup_{n\geq m}\,\bigcup_{\mbox{cell center }x_{j}^{(4n^{2})}}\,\mathcal{P}_{n}^{x_{j}^{(4n^{2})}},
\]
and the probability of the RHS is no greater than 
\[
\lim_{m\rightarrow\infty}\sum_{n\geq m}\,2^{4\nu n^{2}}e^{-n^{3}/M}=0,
\]
$PT_{\theta}^{\gamma}$ is the empty set with probability one.
\end{proof}

\section{Proof of the Lower Bound}

In this section we will provide a proof of the lower bound (\ref{eq:lower bound}).
The strategy is to study the convergence of $\bar{\theta}_{t}\left(x\right)/D\left(t\right)$
as $t\searrow0$ along a prefixed sequence that decays to zero sufficiently
fast. To be specific, assume that $\left\{ r_{n}:\, n\geq0\right\} $
is a sequence of positive numbers satisfying that $r_{0}=1$, $r_{n}\searrow0$
as $n\nearrow\infty$, and 
\begin{equation}
\lim_{n\rightarrow\infty}\frac{n^{2}\cdot\ln r_{n-1}}{\ln r_{n}}=0.\label{eq:condition on r_n}
\end{equation}

\begin{defn}
Let $\gamma\geq0$. For each $\theta\in\Theta$, $x\in\overline{S\left(O,1\right)}$
is called a \emph{sequential} $\gamma-$\emph{thick point }of $\theta$
with the sequence $\left\{ r_{n}:n\ge0\right\} $ if
\[
\lim_{n\rightarrow\infty}\,\frac{\bar{\theta}_{r_{n}}\left(x\right)}{\sqrt{-G\left(r_{n}\right)\ln r_{n}}}=\sqrt{2\nu\gamma}.
\]
With any sequence $\left\{ r_{n}:n\geq0\right\} $ as described above
fixed, we denote by $ST_{\theta}^{\gamma}$ the collection of all
the sequential $\gamma-$thick points of $\theta$ with $\left\{ r_{n}:n\geq0\right\} $.
$ST_{\theta}^{\gamma}$ is a measurable subset of $\overline{S\left(O,1\right)}$.
In this section we will prove the following result.\end{defn}
\begin{thm}
\label{thm:thick point along sequence}For $\gamma\in\left[0,1\right]$,
\[
\dim_{\mathcal{H}}\left(ST_{\theta}^{\gamma}\right)=\nu\left(1-\gamma\right)\mbox{ a.s.}.
\]
Moreover, for $\mathcal{W}-$a.e. $\theta\in\Theta$, $x\in ST_{\theta}^{0}$
for Lebesgue-a.e. $x\in\overline{S\left(O,1\right)}$. 

On the other hand, for $\gamma>1$, $ST_{\theta}^{\gamma}=\emptyset$
a.s..
\end{thm}
Since $ST_{\theta}^{\gamma}\subseteq T_{\theta}^{\gamma}$, the established
upper bounds of the size of $T_{\theta}^{\gamma}$ also apply to $ST_{\theta}^{\gamma}$,
i.e., $\dim_{\mathcal{H}}\left(ST_{\theta}^{\gamma}\right)\leq\nu\left(1-\gamma\right)$
a.s. for $\gamma\in\left[0,1\right]$, and $ST_{\theta}^{\gamma}=\emptyset$
a.s. for $\gamma>1$. As for the lower bound, when $\gamma=0$, (\ref{eq:LIL for concentric})
implies that $\mathcal{W}\left(x\in ST_{\theta}^{0}\right)=1$ for
every $x\in\overline{S\left(O,1\right)}$. Let $\mu_{Leb}$ be the
Lebesgue measure on $\mathbb{R}^{\nu}$, and $\mathcal{H}^{\nu}$
the $\nu-$dimensional Hausdorff measure on $\mathbb{R}^{\nu}$. Then,
$\mathcal{H}^{\nu}=C_{\nu}\mu_{Leb}$ for a dimensional constant $C_{\nu}>0$.
By Fubini's theorem, 
\[
\begin{split}\mathbb{E}^{\mathcal{W}}\left[\mathcal{H}^{\nu}\left(ST_{\theta}^{0}\right)\right] & =C_{\nu}\int_{\overline{S\left(O,1\right)}}\mathcal{W}\left(x\in ST_{\theta}^{0}\right)\mu_{Leb}\left(dx\right)=\mathcal{H}^{\nu}\left(\overline{S\left(O,1\right)}\right).\end{split}
\]
Since $\mathcal{H}^{\nu}\left(ST_{\theta}^{0}\right)\leq\mathcal{H}^{\nu}\left(\overline{S\left(O,1\right)}\right)$
a.s., they must be equal a.s., which implies that for $\mu_{Leb}-$a.e.
$x\in\overline{S\left(O,1\right)}$, $x\in ST_{\theta}^{0}$ and hence
$x\in T_{\theta}^{0}$. Thus, it is sufficient to derive the lower
bound of $\dim_{\mathcal{H}}\left(ST_{\theta}^{\gamma}\right)$ for
$\gamma\in\left(0,1\right)$. 
\begin{rem}
One example of a sequence satisfying (\ref{eq:condition on r_n})
is $r_{n}=2^{-2^{n^{2}}+1}$ for $n\geq0$. However, the method explained
in this section applies to any sufficiently fast decaying sequence.
On the other hand, for technical reasons, we will assume that 
\begin{equation}
\ln\left(-\ln r_{n+1}\right)=o\left(-\ln r_{n}\right)\mbox{ for all large }n\mbox{'s}.\label{eq:condition on r_n (2)}
\end{equation}
This assumption will not reduce the generality of the method. If a
given sequence $\left\{ r_{n}:n\geq0\right\} $ does not satisfy (\ref{eq:condition on r_n (2)}),
one can always ``fill in'' more numbers to get a new sequence $\left\{ \tilde{r}_{m}:m\geq0\right\} $
that satisfied both (\ref{eq:condition on r_n}) and (\ref{eq:condition on r_n (2)}),
and the original sequence $\left\{ r_{n}:n\geq0\right\} $ is a subsequence
of $\left\{ \tilde{r}_{m}:m\geq0\right\} $. Then, if we establish
a lower bound of $\dim_{\mathcal{H}}\left(ST_{\theta}^{\gamma}\right)$
with $\left\{ \tilde{r}_{m}:m\geq0\right\} $, the lower bound also
applies with any subsequence of $\left\{ \tilde{r}_{m}:m\geq0\right\} $. 

The advantage of studying sequential thick points is that the same
method can be applied to the study of other problems related to the
geometry of GFFs, when convergence along sequence already gives rise
to interesting objects (e.g., random measures concerned in \cite{CJ,DS1,JJRV}),
especially in the absence of the perfect $\gamma-$thick point as
pointed out in Theorem \ref{thm: no perfect thick piont}.
\end{rem}
Let $\gamma\in\left(0,1\right)$ be fixed. We will obtain the lower
bound of $\dim_{\mathcal{H}}\left(ST_{\theta}^{\gamma}\right)$ in
multiple steps. To simplify the notation, for every $\theta\in\Theta$,
$x\in\overline{S\left(O,1\right)}$ and $n\geq1$, we write 
\[
\Delta\bar{\theta}_{n}\left(x\right):=\bar{\theta}_{r_{n}}\left(x\right)-\bar{\theta}_{r_{n-1}}\left(x\right),
\]
\[
\Delta G_{n}:=G\left(r_{n}\right)-G\left(r_{n-1}\right)\mbox{ and }\Delta D_{n}:=D\left(r_{n}\right)-D\left(r_{n-1}\right).
\]
For each $x\in\overline{S\left(O,1\right)}$, we define the following
measurable subsets of $\Theta$:
\[
P_{x,0}:=\left\{ \theta\in\Theta:\,\left|\bar{\theta}_{0}\left(x\right)-\sqrt{2\nu\gamma}D\left(r_{0}\right)\right|\leq\sqrt{G\left(r_{0}\right)}\right\} ,
\]
and for $n\ge1$, 
\[
P_{x,n}:=\left\{ \theta\in\Theta:\,\left|\Delta\bar{\theta}_{n}\left(x\right)-\sqrt{2\nu\gamma}\Delta D_{n}\right|\leq\sqrt{\Delta G_{n}}\right\} \mbox{ and }\Phi_{x,n}:=\left(\bigcap_{i=0}^{n}P_{x,i}\right).
\]

\subsection*{Step 1: Derive the probability estimates.}

Let $x\in\overline{S\left(O,1\right)}$ be fixed. It's clear that
$\left\{ \bar{\theta}_{0}\left(x\right),\,\Delta\bar{\theta}_{n}\left(x\right),\, n\geq1\right\} $
is a family of independent Gaussian random variables. The following
simple facts about $P_{x,n}$ and $\Phi_{x,n}$ are in order.
\begin{lem}
\label{lem:analysis of the probability}$P_{x,i}$, $i=0,1,\cdots,n$,
are mutually independent. Moreover, there exists a constant $C_{\nu}>0$
such that for every $n\geq1$, 
\begin{equation}
e^{\nu\gamma\ln r_{n}-C_{\nu}\sqrt{-\ln r_{n}}}\leq\mathcal{W}\left(P_{x,n}\right)\leq e^{\nu\gamma\ln r_{n}+C_{\nu}\sqrt{-\ln r_{n}}}\label{eq: estimate for P_x,n}
\end{equation}
and 
\begin{equation}
e^{\nu\gamma\ln r_{n}\left(1+C_{\nu}/n\right)}\leq\mathcal{W}\left(\Phi_{x,n}\right)\leq e^{\nu\gamma\ln r_{n}\left(1-C_{\nu}/n\right)}.\label{eq: estimate for Phi_x,n}
\end{equation}

\end{lem}
The results above follow from straightforward computations with Gaussian
distributions, combined with the assumption (\ref{eq:condition on r_n}).
Proofs are omitted.

\subsection*{Step 2: Obtain a subset of $ST_{\theta}^{\gamma}$.}

For every $n\ge0$, consider the lattice partition of $\overline{S\left(O,1\right)}$
with cell size $r_{n}$. Assume that $\mathcal{K}_{n}=\left\{ x_{j}^{\left(n\right)}:\, j=1,\cdots,K_{n}\right\} $
is the collection of all the cell centers, where $K_{n}=r_{n}^{-\nu}$.
For every $\theta\in\Theta$, set 
\[
\Xi_{n,\theta}:=\left\{ x_{j}^{(n)}\in\mathcal{K}_{n}:\;1\leq j\leq K_{n},\,\theta\in\Phi_{x_{j}^{(n)},n}\right\} .
\]

\begin{lem}
For every $\gamma\in\left(0,1\right)$ and every $\theta\in\Theta$,
\begin{equation}
ST_{\theta}^{\gamma}\supseteq\Sigma_{\theta}^{\gamma}:=\bigcap_{k\geq1}\,\overline{\bigcup_{n\geq k}\,\bigcup_{x\in\Xi_{n,\theta}}\, S\left(x,r_{n}\right)}.\label{eq:subset of ST}
\end{equation}
.\end{lem}
\begin{proof}
Let $\theta$ and $\gamma$ be fixed. We first show that
\[
ST_{\theta}^{\gamma}\supseteq\bigcap_{k\geq1}\,\bigcup_{n\geq k}\,\bigcup_{x\in\Xi_{n,\theta}}\,\overline{S\left(x,r_{n}\right)}.
\]
 For any $y$ in the RHS above, there exists a subsequence $\left\{ n_{k}:k\geq1\right\} \subseteq\mathbb{N}$
with $n_{k}\nearrow\infty$ as $k\nearrow\infty$ and a sequence of
cell centers $\left\{ x^{(n_{k})}\in\Xi_{n_{k},\theta}:\, k\geq1\right\} $
such that $\left|y-x^{(n_{k})}\right|\leq\sqrt{\nu}\, r_{n_{k}}$
for every $k\geq1$. Moreover, by the definition of $\Xi_{n_{k},\theta}$
and the triangle inequality, for every $j=0,1,\cdots,n_{k}$,
\[
\left|\frac{\bar{\theta}_{r_{j}}\left(x^{\left(n_{k}\right)}\right)}{D\left(r_{j}\right)}-\sqrt{2\nu\gamma}\right|\leq\frac{\sqrt{G\left(r_{0}\right)}+\sum_{p=1}^{j}\sqrt{\Delta G_{p}}}{D\left(r_{j}\right)}\leq\frac{j+1}{\sqrt{-\ln r_{j}}}.
\]
When $j$ is sufficiently large, the RHS above can be arbitrarily
small; moreover, (\ref{eq:almost sure modulus}) implies that, if
$n_{k}$ is large such that $r_{n_{k}}<r_{j+1}^{4}$, then 
\[
\left|\frac{\bar{\theta}_{r_{j}}\left(x^{\left(n_{k}\right)}\right)}{D\left(r_{j}\right)}-\frac{\bar{\theta}_{r_{j}}\left(y\right)}{D\left(r_{j}\right)}\right|\leq r_{j-1}^{3/16}.
\]
It follows immediately from the triangle inequality that 
\[
\lim_{j\rightarrow\infty}\frac{\bar{\theta}_{r_{j}}\left(y\right)}{D\left(r_{j}\right)}=\sqrt{2\nu\gamma},
\]
 and hence $y\in ST_{\theta}^{\gamma}$. 

Next, let $\tilde{y}\in\Sigma_{\theta}^{\gamma}$. For each $k\geq1$,
there exists a sequence $\left\{ y_{p}:\, p\geq1\right\} $ with 
\[
y_{p}\in\bigcup_{n\geq k}\,\bigcup_{x\in\Xi_{n,\theta}}\, S\left(x,r_{n}\right)\,\mbox{ for every }p\geq1
\]
such that $\lim_{p\rightarrow\infty}y_{p}=\tilde{y}$. Either, for
some $n\geq k$, $y_{p}\in\bigcup_{x\in\Xi_{n,\theta}}\, S\left(x,r_{n}\right)$
for infinitely many $p$'s, in which case there must exist $x^{(n)}\in\Xi_{n,\theta}$
such that $\left|\tilde{y}-x^{(n)}\right|\leq2\sqrt{\nu}\cdot r_{n}$,
or, one can find a subsequence $\left\{ n_{p}:\, p\geq0\right\} $
with $n_{p}\nearrow\infty$ as $p\nearrow\infty$ such that $y_{p}\in S\left(x^{(n_{p})},r_{n_{p}}\right)$
for some $x^{(n_{p})}\in\Xi_{n_{p},\theta}$, in which case, since
$y_{p}\rightarrow\tilde{y}$, $\left|x^{(n_{p})}-\tilde{y}\right|$
can be arbitrarily small when $p$ is sufficiently large. In either
case, one can follow similar arguments as above to show that $\tilde{y}\in ST_{\theta}^{\gamma}$. 
\end{proof}

\subsection*{Step 3: Construct a sequence of measures.}

For each $n\geq1$ and $\theta\in\Theta$, define a finite measure
on $\overline{S\left(O,1\right)}$ by, 
\begin{equation}
\forall A\in\mathcal{B}\left(\overline{S\left(O,1\right)}\right),\,\mu_{n,\theta}\left(A\right):=\frac{1}{K_{n}}\sum_{j=1}^{K_{n}}\frac{\mathbb{I}_{\Xi_{n,\theta}}\left(x_{j}^{(n)}\right)}{\mathcal{W}\left(\Phi_{x_{j}^{(n)},n}\right)}\frac{\mbox{vol}\left(A\cap S\left(x_{j}^{(n)},r_{n}\right)\right)}{\mbox{vol}\left(S\left(x_{j}^{(n)},r_{n}\right)\right)}\label{eq:def of mu_n_theta}
\end{equation}
where ``vol'' refers to the volume under the Lebesgue measure on
$\mathbb{R}^{\nu}$. It is clear that 
\begin{equation}
\begin{split}\mathbb{E}^{\mathcal{W}}\left[\mu_{n,\theta}\left(\overline{S\left(O,1\right)}\right)\right] & =1\end{split}
\label{eq:1st moment of total mass under mu_theta}
\end{equation}
for every $n\geq1$. We also need to study the second moment of $\mu_{n,\theta}\left(\overline{S\left(O,1\right)}\right)$,
to which end we write the second moment as 
\begin{equation}
\begin{split}\mathbb{E}^{\mathcal{W}}\left[\left(\mu_{n,\theta}\left(\overline{S\left(O,1\right)}\right)\right)^{2}\right] & =\frac{1}{K_{n}^{2}}\,\sum_{j,k=1}^{K_{n}}\,\frac{\mathcal{W}\left(\Phi_{x_{j}^{(n)},n}\bigcap\Phi_{x_{k}^{(n)},n}\right)}{\mathcal{W}\left(\Phi_{x_{j}^{(n)},n}\right)\mathcal{W}\left(\Phi_{x_{k}^{(n)},n}\right)}\end{split}
.\label{eq:2nd moment of the total mass under mu_theta}
\end{equation}

We will show that 
\[
\sup_{n\geq1}\,\mathbb{E}^{\mathcal{W}}\left[\left(\mu_{n,\theta}\left(\overline{S\left(O,1\right)}\right)\right)^{2}\right]<\infty.
\]
First notice that, when $j=k$, (\ref{eq: estimate for Phi_x,n})
implies that 
\[
\frac{\mathcal{W}\left(\Phi_{x_{j}^{(n)},n}\bigcap\Phi_{x_{k}^{(n)},n}\right)}{\mathcal{W}\left(\Phi_{x_{j}^{(n)},n}\right)\mathcal{W}\left(\Phi_{x_{k}^{(n)},n}\right)}=\frac{1}{\mathcal{W}\left(\Phi_{x_{j}^{(n)},n}\right)}\leq e^{-\nu\gamma\ln r_{n}\left(1+C_{\nu}/n\right)},
\]
so the sum over the diagonal terms in (\ref{eq:2nd moment of the total mass under mu_theta})
is bounded from above by 
\[
K_{n}^{-1}\cdot e^{-\nu\gamma\ln r_{n}\left(1+C_{\nu}/n\right)}=e^{\left(\nu-\nu\gamma\right)\ln r_{n}+o\left(\ln r_{n}\right)}
\]
which converges to zero as $n\rightarrow\infty$ so long as $\gamma<1$.
So we only need to treat the sum over the off-diagonal terms in (\ref{eq:2nd moment of the total mass under mu_theta}),
and this is done in separate cases depending on the distance between
the two cell centers $x_{j}^{(n)}$ and $x_{k}^{(n)}$. 

Assume that $j\neq k$. Then there exists a unique $i\in\mathbb{N}$,
$0\leq i\leq n-1$, such that 
\[
2r_{i+1}\leq\left|x_{j}^{(n)}-x_{k}^{(n)}\right|<2r_{i},\quad\left(\dagger\right)
\]
we can rewrite the sum over the off-diagonal terms in (\ref{eq:2nd moment of the total mass under mu_theta})
as 
\begin{equation}
\frac{1}{K_{n}^{2}}\,\sum_{j=1}^{K_{n}}\,\sum_{i=0}^{n-1}\sum_{\left\{ k:\,(\dagger)\mbox{ holds with }i\right\} }\,\frac{\mathcal{W}\left(\Phi_{x_{j}^{(n)},n}\bigcap\Phi_{x_{k}^{(n)},n}\right)}{\mathcal{W}\left(\Phi_{x_{j}^{(n)},n}\right)\mathcal{W}\left(\Phi_{x_{k}^{(n)},n}\right)}\,.\label{eq:2nd moment rewritten in terms of L}
\end{equation}
Let $j$ and $k$ be fixed for now. For $l,\, l^{\prime}\ge1$, set
\[
\mbox{DCov}\left(l,\, l^{\prime}\right):=\mathbb{E}^{\mathcal{W}}\left[\Delta\bar{\theta}_{l}\left(x_{j}^{(n)}\right)\cdot\Delta\bar{\theta}_{l^{\prime}}\left(x_{k}^{(n)}\right)\right].
\]
By (\ref{eq:covariance for (1-Delta)^s-1}), $\mbox{DCov}\left(l,\, l^{\prime}\right)$
only depends on $r_{l}$, $r_{l^{\prime}}$ and $\left|x_{j}^{(n)}-x_{k}^{(n)}\right|$.
It is sufficient to treat the cases when $\left|x_{j}^{(n)}-x_{k}^{(n)}\right|$
is small, or equivalently, when $i$, as determined by $\left(\dagger\right)$,
is large. One can easily use (\ref{eq:cov non-concentric non-overlapping})
and (\ref{eq:cov non-concentric inclusion}) to verify that $\mbox{DCov}\left(l,\, l^{\prime}\right)=0$
when $l^{\prime}\geq i+2$ and either $l\geq i+2$ or $l\leq i-1$,
which implies that the family 
\begin{equation}
\left\{ \Delta\bar{\theta}_{l}\left(x_{j}^{(n)}\right),\,\Delta\bar{\theta}_{l^{\prime}}\left(x_{k}^{(n)}\right):\,1\leq l\leq i-1,\, i+2\leq l\leq n,\, i+2\leq l^{\prime}\leq n\right\} \label{eq:independent family}
\end{equation}
is independent. However, the independence of this family is not sufficient
for (\ref{eq:2nd moment rewritten in terms of L}) to be bounded in
$n$. We need to carry out more careful analysis by further breaking
down the range of $\left|x_{j}^{(n)}-x_{k}^{(n)}\right|.$

\subsubsection*{Case 1.}

Assume that, for some sufficiently small $\epsilon\in\left(0,1-\gamma\right)$,
\[
2r_{i+1}\leq\left|x_{j}^{(n)}-x_{k}^{(n)}\right|<r_{i+1}^{1-\varepsilon}.
\]
In this case, besides the family of independent random variables in
(\ref{eq:independent family}), we also have that for $l^{\prime}\geq i+2$
, 
\[
B\left(x_{k}^{(n)},r_{l^{\prime}-1}\right)\subseteq B\left(x_{j}^{(n)},r_{i}\right)\mbox{ and }B\left(x_{j}^{(n)},r_{i+1}\right)\bigcap B\left(x_{k}^{(n)},r_{l^{\prime}-1}\right)=\emptyset,
\]
which, by (\ref{eq:cov non-concentric non-overlapping}) and (\ref{eq:cov non-concentric inclusion}),
leads to $\mbox{DCov}\left(i;\, l^{\prime}\right)=0$ and $\mbox{DCov}\left(i+1,\, l^{\prime}\right)=0$,
and hence $\Delta\bar{\theta}_{i}\left(x_{j}^{(n)}\right)$ and $\Delta\bar{\theta}_{i+1}\left(x_{j}^{(n)}\right)$
are independent of $\Delta\bar{\theta}_{l^{\prime}}\left(x_{k}^{(n)}\right)$.
As a result, 
\[
\begin{split}\frac{\mathcal{W}\left(\Phi_{x_{j}^{(n)},n}\bigcap\Phi_{x_{k}^{(n)},n}\right)}{\mathcal{W}\left(\Phi_{x_{j}^{(n)},n}\right)\mathcal{W}\left(\Phi_{x_{k}^{(n)},n}\right)} & \leq\frac{\mathcal{W}\left(\Phi_{x_{j}^{(n)},n}\right)\cdot\Pi_{l^{\prime}=i+2}^{n}\mathcal{W}\left(P_{x_{k}^{(n)},l^{\prime}}\right)}{\mathcal{W}\left(\Phi_{x_{j}^{(n)},n}\right)\mathcal{W}\left(\Phi_{x_{k}^{(n)},n}\right)}\\
 & =\frac{1}{\mathcal{W}\left(\Phi_{x_{k}^{(n)},i+1}\right)}\leq\exp\left[-\nu\gamma\ln r_{i+1}\left(1+C_{\nu}/n\right)\right].
\end{split}
\]
The last inequality is due to (\ref{eq: estimate for Phi_x,n}). On
the other hand, if $j$ is fixed, then the number of $x_{k}^{(n)}$'s
such that 
\[
2r_{i+1}\leq\left|x_{j}^{(n)}-x_{k}^{(n)}\right|<r_{i+1}^{1-\varepsilon}
\]
is of the order of $\left(r_{i+1}^{1-\epsilon}/r_{n}\right)^{\nu}.$
The contribution to (\ref{eq:2nd moment rewritten in terms of L})
under this case is
\[
\begin{split} & \sum_{i=0}^{n-1}\exp\left[\nu\left(1-\epsilon-\gamma\right)\ln r_{i+1}+o\left(-\ln r_{i+1}\right)\right]\end{split}
\]
which is bounded in $n$ since $\epsilon<1-\gamma$.

\subsubsection*{Case 2.}

Assume that
\begin{equation}
r_{i}-r_{i+2}<\left|x_{j}^{(n)}-x_{k}^{(n)}\right|\leq r_{i}+r_{i+2}.\label{eq:case 2}
\end{equation}
Since the random variables in (\ref{eq:independent family}) are independent,
we have that 
\[
\begin{split}\frac{\mathcal{W}\left(\Phi_{x_{j}^{(n)},n}\bigcap\Phi_{x_{k}^{(n)},n}\right)}{\mathcal{W}\left(\Phi_{x_{j}^{(n)},n}\right)\mathcal{W}\left(\Phi_{x_{k}^{(n)},n}\right)} & \leq\frac{\mathcal{W}\left(\Phi_{x_{j}^{(n)},i-1}\right)\Pi_{l=i+2}^{n}\mathcal{W}\left(P_{x_{j}^{(n)},l}\right)\Pi_{l^{\prime}=i+2}^{n}\mathcal{W}\left(P_{x_{k}^{(n)},l^{\prime}}\right)}{\mathcal{W}\left(\Phi_{x_{j}^{(n)},n}\right)\mathcal{W}\left(\Phi_{x_{k}^{(n)},n}\right)}\\
 & =\frac{1}{\mathcal{W}\left(P_{x_{j}^{(n)},i}\right)\mathcal{W}\left(P_{x_{j}^{(n)},i+1}\right)\mathcal{W}\left(\Phi_{x_{k}^{(n)},i+1}\right)},
\end{split}
\]
which, by (\ref{eq: estimate for P_x,n}) and (\ref{eq: estimate for Phi_x,n}),
is no greater than 
\[
\exp\left[-3\nu\gamma\ln r_{i+1}\left(1+C_{\nu}/n\right)\right].
\]
Meanwhile, with $x_{j}^{(n)}$ fixed, the number of $x_{k}^{(n)}$'s
that satisfy (\ref{eq:case 2}) is of the order of $r_{i}^{\nu-1}r_{i+2}/r_{n}^{\nu}$.
Hence, the contribution to (\ref{eq:2nd moment rewritten in terms of L})
under this case is
\[
\begin{split} & \sum_{i=0}^{n-1}\exp\left[\ln r_{i+2}-3\nu\gamma\ln r_{i+1}+o\left(-\ln r_{i+1}\right)\right]\end{split}
\]
which is bounded in $n$ by the assumption (\ref{eq:condition on r_n}).

\subsubsection*{Case 3}

Assume that either 
\[
r_{i}-r_{i+1}<\left|x_{j}^{(n)}-x_{k}^{(n)}\right|\leq r_{i}-r_{i+2}\quad\mbox{(3a)}
\]
or 
\[
r_{i}+r_{i+2}\leq\left|x_{j}^{(n)}-x_{k}^{(n)}\right|<r_{i}+r_{i+1}.\quad\mbox{(3b)}
\]
We observe that for all $l^{\prime}\geq i+3$, by (\ref{eq:cov non-concentric non-overlapping})
and (\ref{eq:cov non-concentric inclusion}), under the hypothesis
(3a) or (3b), $\mbox{DCov}\left(i+1,\, l^{\prime}\right)=0$. Together
with the family of independent random variables in (\ref{eq:independent family}),
we see that both $P_{x_{j}^{(n)},i+1}$ and $P_{x_{j}^{(n)},i+2}$
are independent of $P_{x_{k}^{(n)},l^{\prime}}$ for all $l^{\prime}\geq i+3$,
and similarly both $P_{x_{k}^{(n)},i+1}$ and $P_{x_{k}^{(n)},i+2}$
are independent of $P_{x_{j}^{(n)},l}$ for all $l\geq i+3$. Thus,
$\mathcal{W}\left(\Phi_{x_{j}^{(n)},n}\bigcap\Phi_{x_{k}^{(n)},n}\right)$
is bounded from above by
\begin{equation}
\begin{split} & \Pi_{l=i+3}^{n}\mathcal{W}\left(P_{x_{j}^{(n)},l}\right)\cdot\Pi_{l^{\prime}=i+3}^{n}\mathcal{W}\left(P_{x_{k}^{(n)},l^{\prime}}\right)\\
 & \qquad\qquad\qquad\cdot\mathcal{W}\left(P_{x_{j}^{(n)},i+1}\cap P_{x_{j}^{(n)},i+2}\cap P_{x_{k}^{(n)},i+1}\cap P_{x_{k}^{(n)},i+2}\right),
\end{split}
\label{eq:splitting prob in case 3}
\end{equation}
so we only need to focus on the family 
\[
\left\{ \Delta\bar{\theta}_{i+1}\left(x_{j}^{(n)}\right),\Delta\bar{\theta}_{i+2}\left(x_{j}^{(n)}\right),\Delta\bar{\theta}_{i+1}\left(x_{k}^{(n)}\right),\Delta\bar{\theta}_{i+2}\left(x_{k}^{(n)}\right)\right\} .
\]

\begin{lem}
\label{lem:estimate of prob of i+1,i+2,i+1,i+2}Under the hypothesis
(3a) or (3b), there exists a constant $C_{\nu}>0$ such that for all
$i\geq0$, 
\begin{equation}
\begin{split} & \frac{\mathcal{W}\left(P_{x_{j}^{(n)},i+1}\cap P_{x_{j}^{(n)},i+2}\cap P_{x_{k}^{(n)},i+1}\cap P_{x_{k}^{(n)},i+2}\right)}{\mathcal{W}\left(P_{x_{j}^{(n)},i+1}\right)\mathcal{W}\left(P_{x_{j}^{(n)},i+2}\right)\mathcal{W}\left(P_{x_{k}^{(n)},i+1}\right)\mathcal{W}\left(P_{x_{k}^{(n)},i+2}\right)}\leq e^{C_{\nu}\sqrt{-\ln r_{i+1}}}.\end{split}
\label{eq:estimate of prob of i+1, i+2, i+1, i+2}
\end{equation}
\end{lem}
\begin{proof}
We will prove this result by multiple steps of conditioning. To further
simplify the notation, throughout the proof, we write 
\[
X_{i+1}:=\Delta\bar{\theta}_{i+1}\left(x_{j}^{(n)}\right),\; X_{i+2}:=\Delta\bar{\theta}_{i+2}\left(x_{j}^{(n)}\right),
\]
and 
\[
Y_{i+1}:=\Delta\bar{\theta}_{i+1}\left(x_{k}^{(n)}\right),\; Y_{i+2}:=\Delta\bar{\theta}_{i+2}\left(x_{k}^{(n)}\right).
\]
Clearly, $Y_{i+2}$ is independent of $Y_{i+1}$ and $X_{i+2}$. Furthermore,
$\mbox{Cov}\left(X_{i+1},Y_{i+2}\right)$ is given by, when (3a) applies,
\[
\begin{split}\mbox{DCov}\left(i+1,\, i+2\right) & =-C_{incl}\left(r_{i},\left|x_{j}^{(n)}-x_{k}^{(n)}\right|\right)+\mbox{Cov}\left(r_{i},x_{j}^{(n)};r_{i+1},x_{k}^{(n)}\right);\end{split}
\]
when (3b) applies,
\[
\begin{split}\mbox{DCov}\left(i+1,\, i+2\right) & =-C_{disj}\left(\left|x_{j}^{(n)}-x_{k}^{(n)}\right|\right)+\mbox{Cov}\left(r_{i},x_{j}^{(n)};r_{i+1},x_{k}^{(n)}\right).\end{split}
\]
In either case, $\mbox{Cov}\left(X_{i+1},Y_{i+2}\right)$ doesn't
depend on $r_{i+2}$, and by the asymptotics of the functions that
are involved and the Cauchy-Schwarz inequality, 
\begin{equation}
\mbox{Cov}\left(X_{i+1},Y_{i+2}\right)=\mathcal{O}\left(\sqrt{G\left(r_{i+1}\right)G\left(r_{i}\right)}\right).\label{eq:bound for cov(i+2,i+1)}
\end{equation}
Similarly, $\mbox{Cov}\left(X_{i+1},Y_{i+1}\right)$ is given by,
when either (3a) or (3b) applies,
\[
\begin{split}\mbox{DCov}\left(i+1,\, i+1\right) & =C_{disj}\left(\left|x_{j}^{(n)}-x_{k}^{(n)}\right|\right)-2\mbox{Cov}\left(r_{i},x_{j}^{(n)};r_{i+1},x_{k}^{(n)}\right)\\
 & \hfill\hfill\hfill+\mbox{Cov}\left(r_{i},x_{j}^{(n)};r_{i},x_{k}^{(n)}\right),
\end{split}
\]
which implies that
\begin{equation}
\mbox{Cov}\left(X_{i+1},Y_{i+1}\right)=\mathcal{O}\left(\sqrt{G\left(r_{i+1}\right)G\left(r_{i}\right)}\right).\label{eq:bound for cov(i+1,i+1)}
\end{equation}

We first condition on $Y_{i+2}$. The joint conditional distribution
of $\left\{ X_{i+1},X_{i+2},Y_{i+1}\right\} $, given $Y_{i+2}=y$,
is the same as the Gaussian family $\left\{ X_{i+1}^{\prime},X_{i+2}^{\prime},Y_{i+1}^{\prime}\right\} $
where $X_{i+2}^{\prime}$ and $Y_{i+1}^{\prime}$ have the same distribution
as $X_{i+2}$ and $Y_{i+1}$ respectively, and $X_{i+1}^{\prime}$
has the Gaussian distribution $N\left(m,\sigma^{2}\right)$ with 
\[
m:=\frac{\mbox{Cov}\left(X_{i+1},Y_{i+2}\right)}{\Delta G_{i+2}}y\quad\mbox{ and }\quad\sigma^{2}:=\Delta G_{i+1}-\frac{\mbox{Cov}^{2}\left(X_{i+1},Y_{i+2}\right)}{\Delta G_{i+2}}.
\]
In particular, if $\left|y-\sqrt{2\nu\gamma}\Delta D_{i+2}\right|\leq\sqrt{\Delta G_{i+2}}$,
then, by (\ref{eq:condition on r_n}) and (\ref{eq:bound for cov(i+2,i+1)}),
$m=o\left(1\right)$ and $\sigma^{2}=\Delta G_{i+1}+o\left(1\right)$,
and these estimates%
\footnote{Here, as well as in later occasions, when concerning $o\left(1\right)$,
the ``estimate'' refers to the rate of the $o\left(1\right)$ term
converging to zero. %
} can be made uniform in $y$. Moreover, the covariance of the family
is given by $\mbox{Cov}\left(X_{i+1}^{\prime},X_{i+2}^{\prime}\right)=0$,
$\mbox{Cov}\left(X_{i+2}^{\prime},Y_{i+1}^{\prime}\right)=\mbox{Cov}\left(X_{i+2},Y_{i+1}\right)$
and $\mbox{Cov}\left(X_{i+1}^{\prime},Y_{i+1}^{\prime}\right)=\mbox{Cov}\left(X_{i+1},Y_{i+1}\right)$.
We write the following conditional distribution as
\[
\mathcal{W}\left(P_{x_{j}^{(n)},i+1}\cap P_{x_{j}^{(n)},i+2}\cap P_{x_{k}^{(n)},i+1}|Y_{i+2}=y\right)=\mathcal{W}|_{Y_{i+2}=y}\left(P_{X_{i+1}^{\prime}}\cap P_{X_{i+2}^{\prime}}\cap P_{Y_{i+1}^{\prime}}\right),
\]
where $\mathcal{W}|_{Y_{i+2}=y}$ is the conditional distribution
under $\mathcal{W}$ given $Y_{i+2}=y$, and $P_{X_{i+1}^{\prime}}$,
$P_{X_{i+2}^{\prime}}$ and $P_{Y_{i+1}^{\prime}}$ are the corresponding
events concerning $X_{i+1}^{\prime}$, $X_{i+2}^{\prime}$ and $Y_{i+1}^{\prime}$,
e.g.,
\[
P_{X_{i+1}^{\prime}}=\left\{ \left|X_{i+1}^{\prime}-\sqrt{2\nu\gamma}\Delta D_{i+1}\right|\leq\sqrt{\Delta G_{i+1}}\right\} .
\]

Next, we condition on $X_{i+2}^{\prime}=x$ where $\left|x-\sqrt{2\nu\gamma}\Delta D_{i+2}\right|\leq\sqrt{\Delta G_{i+2}}$.
Then the conditional distribution of $\left\{ X_{i+1}^{\prime},Y_{i+1}^{\prime}\right\} $
is the same as that of $\left\{ X_{i+1}^{\prime\prime},Y_{i+1}^{\prime\prime}\right\} $
where $X_{i+1}^{\prime\prime}$ has the same distribution as $X_{i+1}^{\prime}$,
and $Y_{i+1}^{\prime\prime}$ has the Gaussian distribution $N\left(\lambda,\varsigma^{2}\right)$
where 
\[
\lambda=\frac{\mbox{Cov}\left(X_{i+2},Y_{i+1}\right)}{\Delta G_{i+2}}x\quad\mbox{ and }\quad\varsigma^{2}=\Delta G_{i+1}-\frac{\mbox{Cov}^{2}\left(X_{i+2},Y_{i+1}\right)}{\Delta G_{i+2}}.
\]
Since $\mbox{Cov}\left(X_{i+2},Y_{i+1}\right)=\mbox{Cov}\left(X_{i+1},Y_{i+2}\right)$,
the estimates we obtained for $m$ and $\sigma^{2}$ also apply to
$\lambda$ and $\varsigma^{2}$ respectively, and those estimates
are uniform in $x$ and $y$. In addition, $\mbox{Cov}\left(X_{i+1}^{\prime\prime},Y_{i+1}^{\prime\prime}\right)=\mbox{Cov}\left(X_{i+1},Y_{i+1}\right)$.
Again, we write the following conditional distribution as 
\[
\mathcal{W}|_{Y_{i+2}=y}\left(P_{X_{i+1}^{\prime}}\cap P_{Y_{i+1}^{\prime}}|X_{i+2}^{\prime}=x\right)=\mathcal{W}|_{X_{i+2}^{\prime}=x}\left(P_{X_{i+1}^{\prime\prime}}\cap P_{Y_{i+1}^{\prime\prime}}\right)
\]
where $\mathcal{W}|_{X_{i+2}^{\prime}=x}$ is the conditional distribution
under $\mathcal{W}|_{Y_{i+2}=y}$ conditioning on $X_{i+2}^{\prime}=x$,
and $P_{X_{i+1}^{\prime\prime}}$ and $P_{Y_{i+1}^{\prime\prime}}$
are the corresponding events concerning $X_{i+1}^{\prime\prime}$
and $Y_{i+1}^{\prime\prime}$.

To compute $\mathcal{W}|_{X_{i+2}^{\prime}=x}\left(P_{X_{i+1}^{\prime\prime}}\cap P_{Y_{i+1}^{\prime\prime}}\right)$,
we use conditioning again. Given 
\[
Y_{i+1}^{\prime\prime}=w\in\left[\sqrt{2\nu\gamma}\Delta D_{i+1}-\sqrt{\Delta G_{i+1}},\,\sqrt{2\nu\gamma}\Delta D_{i+1}+\sqrt{\Delta G_{i+1}}\right],
\]
the conditional distribution of $X_{i+1}^{\prime\prime}$ is the Gaussian
distribution with the mean 
\[
m+\frac{\mbox{Cov}\left(X_{i+1},Y_{i+1}\right)}{\varsigma^{2}}\left(w-\lambda\right)=\mathcal{O}\left(\sqrt{G\left(r_{i}\right)\left(-\ln r_{i+1}\right)}\right)
\]
and the variance 
\[
\sigma^{2}-\frac{\mbox{Cov}^{2}\left(X_{i+1},Y_{i+1}\right)}{\varsigma^{2}}=\Delta G_{i+1}\left(1+o\left(1\right)\right).
\]
These estimates%
\footnote{Here, as well as in later occasions, when concerning ``$\mathcal{O}$'',
the ``estimate'' refers to the constants in the upper and lower
bound.%
} follow from (\ref{eq:bound for cov(i+1,i+1)}) and earlier estimates
on $m$, $\lambda$, $\sigma^{2}$ and $\varsigma^{2}$, and they
can be made uniform in $w$, $x$ and $y$. Therefore, one can easily
verify that 
\[
\begin{split}\mathcal{W}|_{X_{i+2}^{\prime}=x}\left(P_{X_{i+1}^{\prime\prime}}|Y_{i+1}^{\prime\prime}=w\right) & \leq\exp\left[\nu\gamma\ln r_{i+1}+\mathcal{O}\left(\sqrt{-\ln r_{i+1}}\right)\right]:=p_{1},\end{split}
\]
and $p_{1}$ is independent of $w$, $x$ and $y$. This further leads
to 
\[
\begin{split}\mathcal{W}|_{X_{i+2}^{\prime}=x}\left(P_{X_{i+1}^{\prime\prime}}\cap P_{Y_{i+1}^{\prime\prime}}\right) & \leq p_{1}\exp\left[\nu\gamma\ln r_{i+1}+\mathcal{O}\left(\sqrt{-\ln r_{i+1}}\right)\right]\\
 & =\exp\left[2\nu\gamma\ln r_{i+1}+\mathcal{O}\left(\sqrt{-\ln r_{i+1}}\right)\right]:=p_{2},
\end{split}
\]
and $p_{2}$ is independent of $x$ and $y$.

Finally, since $X_{i+2}^{\prime}$ has the same distribution as $X_{i+2}$,
by backtracking the condition, we have that
\[
\mathcal{W}|_{Y_{i+2}=y}\left(P_{X_{i+1}^{\prime}}\cap P_{Y_{i+1}^{\prime}}\cap P_{X_{i+2}^{\prime}}\right)\leq p_{2}\mathcal{W}\left(P_{x_{j}^{(n)},i+2}\right),
\]
and hence 
\[
\mathcal{W}\left(P_{x_{j}^{(n)},i+1}\cap P_{x_{j}^{(n)},i+2}\cap P_{x_{k}^{(n)},i+1}\cap P_{x_{k}^{(n)},i+2}\right)\leq p_{2}\mathcal{W}\left(P_{x_{j}^{(n)},i+2}\right)\mathcal{W}\left(P_{x_{k}^{(n)},i+2}\right).
\]
The desired estimate follows immediately from (\ref{eq: estimate for P_x,n}). 
\end{proof}
It follows from (\ref{eq:condition on r_n}), (\ref{eq:splitting prob in case 3})
and (\ref{eq:estimate of prob of i+1, i+2, i+1, i+2}) that, in Case
3, 
\[
\begin{split}\frac{\mathcal{W}\left(\Phi_{x_{j}^{(n)},n}\bigcap\Phi_{x_{k}^{(n)},n}\right)}{\mathcal{W}\left(\Phi_{x_{j}^{(n)},n}\right)\mathcal{W}\left(\Phi_{x_{k}^{(n)},n}\right)} & \leq\frac{\mathcal{W}\left(P_{x_{j}^{(n)},i+1}\cap P_{x_{j}^{(n)},i+2}\cap P_{x_{k}^{(n)},i+1}\cap P_{x_{k}^{(n)},i+2}\right)}{\mathcal{W}\left(\Phi_{x_{j}^{(n)},i+2}\right)\mathcal{W}\left(\Phi_{x_{k}^{(n)},i+2}\right)}\\
 & \leq\frac{\exp\left(C_{\nu}\sqrt{-\ln r_{i+1}}\right)}{\mathcal{W}\left(\Phi_{x_{j}^{(n)},i}\right)\mathcal{W}\left(\Phi_{x_{k}^{(n)},i}\right)}=\exp\left[o\left(-\ln r_{i+1}\right)\right].
\end{split}
\]
On the other hand, with $x_{j}^{(n)}$ fixed, the number of $x_{k}^{(n)}$'s
that satisfy either (3a) or (3b) is of the order of $r_{i}^{\nu-1}r_{i+1}/r_{n}^{\nu}$.
Hence, the contribution to (\ref{eq:2nd moment rewritten in terms of L})
under this case is
\[
\begin{split} & \sum_{i=0}^{n-1}\exp\left[\ln r_{i+1}+o\left(-\ln r_{i+1}\right)\right]\end{split}
\]
which is bounded in $n$.

\subsubsection*{Case 4.}

The last case is that either 
\[
r_{i+1}^{1-\epsilon}<\left|x_{j}^{(n)}-x_{k}^{(n)}\right|\leq r_{i}-r_{i+1}\quad\mbox{(4a)}
\]
or 
\[
r_{i}+r_{i+1}\leq\left|x_{j}^{(n)}-x_{k}^{(n)}\right|<2r_{i}.\quad\,\mbox{(4b)}
\]
The strategy for studying this case is similar to that for the previous
case. We will omit the technical details that are the same as earlier,
but only address the differences in the treatment of Case 4 from that
of Case 3. When (4a) or (4b) applies, one can use (\ref{eq:cov non-concentric non-overlapping})
and (\ref{eq:cov non-concentric inclusion}) to verify that both $P_{x_{j}^{(n)},i}$
and $P_{x_{j}^{(n)},i+1}$ are independent of $P_{x_{k}^{(n)},l^{\prime}}$
for all $l^{\prime}\geq i+2$, and $P_{x_{k}^{(n)},i+1}$ is independent
of $P_{x_{j}^{(n)},l}$ for all $l\geq i+2$. Thus, $\mathcal{W}\left(\Phi_{x_{j}^{(n)},n}\bigcap\Phi_{x_{k}^{(n)},n}\right)$
is no greater than 
\begin{equation}
\mathcal{W}\left(\bigcap_{l=i+2}^{n}P_{x_{j}^{(n)},l}\right)\mathcal{W}\left(\bigcap_{l^{\prime}=i+2}^{n}P_{x_{k}^{(n)},l^{\prime}}\right)\cdot\mathcal{W}\left(P_{x_{j}^{(n)},i}\cap P_{x_{j}^{(n)},i+1}\cap P_{x_{k}^{(n)},i+1}\right).\label{eq:splitting prob in case 4}
\end{equation}

\begin{lem}
Under the hypothesis (4a) or (4b), there exists a constant $C_{\nu,\epsilon}>0$
such that for all $i\geq0$, 
\begin{equation}
\frac{\mathcal{W}\left(P_{x_{j}^{(n)},i}\cap P_{x_{j}^{(n)},i+1}\cap P_{x_{k}^{(n)},i+1}\right)}{\mathcal{W}\left(P_{x_{j}^{(n)},i}\right)\mathcal{W}\left(P_{x_{j}^{(n)},i+1}\right)\mathcal{W}\left(P_{x_{k}^{(n)},i+1}\right)}\leq\exp\left(C_{\nu,\epsilon}\sqrt{-\ln r_{i}}\right).\label{eq:estimate of prob of i,i+1.i+1}
\end{equation}
\end{lem}
\begin{proof}
Similarly as the proof of Lemma \ref{lem:estimate of prob of i+1,i+2,i+1,i+2},
one can prove (\ref{eq:estimate of prob of i,i+1.i+1}) by multiples
steps of conditioning. For simpler notation, we write 
\[
X_{i}:=\Delta\bar{\theta}_{i}\left(x_{j}^{(n)}\right),\, X_{i+1}:=\Delta\bar{\theta}_{i+1}\left(x_{j}^{(n)}\right),\mbox{ and }Y_{i+1}:=\Delta\bar{\theta}_{i+1}\left(x_{k}^{(n)}\right).
\]
When (4a) or (4b) applies, by (\ref{eq:cov non-concentric non-overlapping})
and (\ref{eq:cov non-concentric inclusion}), 
\begin{equation}
\mbox{Cov}\left(X_{i},Y_{i+1}\right)=\mathcal{O}\left(G\left(r_{i}\right)\right),\label{eq:bound for i,i+1}
\end{equation}
and 
\begin{equation}
\mbox{Cov}\left(Y_{i+1},X_{i+1}\right)=\mathcal{O}\left(G^{1-\epsilon}\left(r_{i+1}\right)\right).\label{eq:bound for i+1,i+1}
\end{equation}
We first condition on $Y_{i+1}=y$ where $\left|y-\sqrt{2\nu\gamma}\Delta D_{i+1}\right|\leq\sqrt{\Delta G_{i+1}}$.
Then the joint conditional distribution of $\left\{ X_{i},X_{i+1}\right\} $
given $Y_{i+1}=y$ is the same as that of $\left\{ X_{i}^{\prime},X_{i+1}^{\prime}\right\} $
where $X_{i}^{\prime}$ and $X_{i+1}^{\prime}$ have distributions
$N\left(m_{1},\sigma_{1}^{2}\right)$ and $N\left(m_{2},\sigma_{2}^{2}\right)$
respectively, where $m_{1}=o\left(1\right)$, $\sigma_{1}^{2}=\Delta G_{i}+o\left(1\right)$,
and 
\begin{equation}
m_{2}=\mathcal{O}\left(\Delta D_{i+1}\cdot G^{-\epsilon}\left(r_{i+1}\right)\right)\mbox{ and }\sigma_{2}^{2}=\Delta G_{i+1}\left[1+\mathcal{O}\left(G^{-2\epsilon}\left(r_{i+1}\right)\right)\right],\label{eq:estimate for m_2, sigma_2}
\end{equation}
and moreover,
\begin{equation}
\mbox{Cov}\left(X_{i}^{\prime},X_{i+1}^{\prime}\right)=\mathcal{O}\left(G\left(r_{i}\right)/G^{\epsilon}\left(r_{i+1}\right)\right).\label{eq:estimate for cov i,i+1}
\end{equation}
These estimates on $m_{1}$, $\sigma_{1}^{2}$, $m_{2}$, $\sigma_{2}^{2}$
and $\mbox{Cov}\left(X_{i}^{\prime},X_{i+1}^{\prime}\right)$ follow
from (\ref{eq:bound for i,i+1}) and (\ref{eq:bound for i+1,i+1})
and can be made uniform in $y$. Next, given $X_{i}^{\prime}=x$ where
$\left|x-\sqrt{2\nu\gamma}\Delta D_{i}\right|\leq\sqrt{\Delta G_{i}}$,
the conditional distribution of $X_{i+1}^{\prime}$ is the Gaussian
distribution $N\left(m_{3},\sigma_{3}^{2}\right)$ and, by (\ref{eq:estimate for cov i,i+1}),
$m_{3}$ and $\sigma_{3}^{2}$ follow the same estimates as $m_{2}$
and $\sigma_{2}^{2}$ respectively, i.e., the estimates in (\ref{eq:estimate for m_2, sigma_2}),
and these estimates are uniform in $x$ and $y$.

To proceed from here, we need to carry out a step that is different
from the proof of Lemma \ref{lem:estimate of prob of i+1,i+2,i+1,i+2}.
Specifically, we need to compare $\mathcal{W}\left(P_{x_{j}^{(n)},i+1}\right)$
and 
\[
N\left(m_{3},\sigma_{3}^{2}\right)\left(\left[\sqrt{2\nu\gamma}\Delta D_{i+1}-\sqrt{\Delta G_{i+1}},\sqrt{2\nu\gamma}\Delta D_{i+1}+\sqrt{\Delta G_{i+1}}\right]\right).
\]
To this end, we write the later as
\[
\begin{split} & \frac{1}{\sqrt{2\pi}}\int_{\sqrt{2\nu\gamma}\Delta D_{i+1}-\sqrt{\Delta G_{i+1}}}^{\sqrt{2\nu\gamma}\Delta D_{i+1}+\sqrt{\Delta G_{i+1}}}\frac{\exp\left[-\left(w-m_{3}\right)^{2}/\left(2\sigma_{3}^{2}\right)\right]}{\sigma_{3}}dw\\
=\; & \frac{1}{\sqrt{2\pi}}\int_{\sqrt{2\nu\gamma}\Delta D_{i+1}-\sqrt{\Delta G_{i+1}}}^{\sqrt{2\nu\gamma}\Delta D_{i+1}+\sqrt{\Delta G_{i+1}}}\,\frac{\exp\left[-w^{2}/\left(2\Delta G_{i+1}\right)\right]}{\sqrt{\Delta G_{i+1}}}\cdot E\left(w\right)dw
\end{split}
\]
 where 
\[
E\left(w\right):=\frac{\sqrt{\Delta G_{i+1}}}{\sigma_{3}}\exp\left[-\frac{\left(w-m_{3}\right)^{2}}{2\sigma_{3}^{2}}+\frac{w^{2}}{2\Delta G_{i+1}}\right].
\]
Notice that by the estimates in (\ref{eq:estimate for m_2, sigma_2})
which apply to $m_{3}$ and $\sigma_{3}^{2}$, there exists a constant
$C_{\nu,\epsilon}>0$ such that 
\[
\begin{split}\sup_{\left\{ w:\left|w-\sqrt{2\nu\gamma}\Delta D_{i+1}\right|\leq\sqrt{\Delta G_{i+1}}\right\} }\,\left|E\left(w\right)\right| & \leq C_{\nu,\epsilon}.\end{split}
\]
It follows from this observation that, given $X_{i}^{\prime}=x$,
the conditional probability of $X_{i+1}^{\prime}$ being in the desired
interval, i.e., $\left|X_{i+1}^{\prime}-\sqrt{2\nu\gamma}\Delta D_{i+1}\right|\leq\sqrt{\Delta G_{i+1}}$,
is bounded by $C_{\nu,\epsilon}\mathcal{W}\left(P_{x_{j}^{(n)},i+1}\right)$.
From this point, we backtrack the conditioning in the same way as
we did in the proof of Lemma \ref{lem:estimate of prob of i+1,i+2,i+1,i+2}
and arrive at 
\[
\mathcal{W}\left(P_{x_{j}^{(n)},i}\cap P_{x_{j}^{(n)},i+1}\cap P_{x_{k}^{(n)},i+1}\right)\leq e^{\nu\gamma\ln r_{i}+\mathcal{O}\left(\sqrt{-\ln r_{i}}\right)}\mathcal{W}\left(P_{x_{j}^{(n)},i+1}\right)\mathcal{W}\left(P_{x_{k}^{(n)},i+1}\right).
\]
By (\ref{eq: estimate for P_x,n}), (\ref{eq:estimate of prob of i,i+1.i+1})
follows immediately. 
\end{proof}
Based on (\ref{eq: estimate for Phi_x,n}), (\ref{eq:splitting prob in case 4})
and (\ref{eq:estimate of prob of i,i+1.i+1}), we have that, in Case
4, 
\[
\begin{split}\frac{\mathcal{W}\left(\Phi_{x_{j}^{(n)},n}\bigcap\Phi_{x_{k}^{(n)},n}\right)}{\mathcal{W}\left(\Phi_{x_{j}^{(n)},n}\right)\mathcal{W}\left(\Phi_{x_{k}^{(n)},n}\right)} & \leq\frac{\mathcal{W}\left(P_{x_{j}^{(n)},i}\cap P_{x_{j}^{(n)},i+1}\cap P_{x_{k}^{(n)},i+1}\right)}{\mathcal{W}\left(\Phi_{x_{j}^{(n)},i+1}\right)\mathcal{W}\left(\Phi_{x_{k}^{(n)},i+1}\right)}\\
 & \leq\frac{\exp\left(C_{\nu,\epsilon}\sqrt{-\ln r_{i}}\right)}{\mathcal{W}\left(\Phi_{x_{j}^{(n)},i-1}\right)\mathcal{W}\left(\Phi_{x_{k}^{(n)},i}\right)}\\
 & \leq\exp\left[-\nu\gamma\ln r_{i}+o\left(-\ln r_{i}\right)\right].
\end{split}
\]
With $x_{j}^{(n)}$ fixed, the number of $x_{k}^{(n)}$'s that satisfy
either (4a) or (4b) is of the order of $\left(2r_{i}\right)^{\nu}/r_{n}^{\nu}$.
Hence, the contribution to (\ref{eq:2nd moment rewritten in terms of L})
under this case is
\[
\begin{split} & \sum_{i=0}^{n-1}\exp\left[\nu\left(1-\gamma\right)\ln r_{i}+o\left(-\ln r_{i}\right)\right]\end{split}
\]
which is bounded in $n$ since $\gamma<1$.

Summarizing our findings in all the cases above, we conclude that
\[
\sup_{n\geq1}\,\mathbb{E}^{\mathcal{W}}\left[\left(\mu_{n,\theta}\left(\overline{S\left(O,1\right)}\right)\right)^{2}\right]<\infty.
\]

\subsection*{Step 4: Study the $\alpha-$energy of $\mu_{n,\theta}$.}

In this subsection we study the $\alpha-$energy, $\alpha>0$, of
the measure $\mu_{n,\theta}$ introduced previously. Namely, with
$\alpha>0$ fixed, we consider, for every $\theta\in\Theta$ and every
$n\geq1$, 
\[
I_{\alpha}\left(\mu_{n,\theta}\right):=\int_{\overline{S\left(O,1\right)}}\int_{\overline{S\left(O,1\right)}}\left|y-w\right|^{-\alpha}\mu_{n,\theta}\left(dy\right)\mu_{n,\theta}\left(dw\right).
\]
By the definition of $\mu_{n,\theta}$ (\ref{eq:def of mu_n_theta}),
$\mathbb{E}^{\mathcal{W}}\left[I_{\alpha}\left(\mu_{n,\theta}\right)\right]$
is equal to 
\begin{equation}
\frac{1}{K_{n}^{2}}\sum_{j,k=1}^{K_{n}}\frac{\mathcal{W}\left(\Phi_{x_{j}^{(n)},n}\bigcap\Phi_{x_{k}^{(n)},n}\right)}{\mathcal{W}\left(\Phi_{x_{j}^{(n)},n}\right)\mathcal{W}\left(\Phi_{x_{k}^{(n)},n}\right)}\frac{\int_{\overline{S\left(x_{j}^{(n)},r_{n}\right)}}\int_{\overline{S\left(x_{k}^{(n)},r_{n}\right)}}\left|y-w\right|^{-\alpha}dydw}{\mbox{vol}\left(S\left(x_{j}^{(n)},r_{n}\right)\right)\mbox{vol}\left(S\left(x_{k}^{(n)},r_{n}\right)\right)}.\label{eq:expection of energy}
\end{equation}
In this subsection we will show that, if $\alpha<\nu\left(1-\gamma\right)$,
then 
\[
\sup_{n\geq1}\,\mathbb{E}^{\mathcal{W}}\left[I_{\alpha}\left(\mu_{n,\theta}\right)\right]<\infty.
\]
For simplicity, we write 
\[
I\left(x_{j}^{(n)},\, x_{k}^{(n)}\right):=\frac{\int_{\overline{S\left(x_{j}^{(n)},r_{n}\right)}}\int_{\overline{S\left(x_{k}^{(n)},r_{n}\right)}}\left|y-w\right|^{-\alpha}dydw}{\mbox{vol}\left(S\left(x_{j}^{(n)},r_{n}\right)\right)\mbox{vol}\left(S\left(x_{k}^{(n)},r_{n}\right)\right)}.
\]

When $j=k$, so long as $\alpha<\nu$, $I\left(x_{j}^{(n)},\, x_{k}^{(n)}\right)=C_{\nu}\cdot r_{n}^{-\alpha}$
for some dimensional constant $C_{\nu}>0$. Therefore, the sum over
the diagonal terms in (\ref{eq:expection of energy}) is 
\[
\begin{split}\frac{1}{K_{n}^{2}}\,\sum_{j=1}^{K_{n}}\frac{C_{\nu}\cdot r_{n}^{-\alpha}}{\mathcal{W}\left(\Phi_{x_{j}^{(n)},n}\right)} & \leq C_{\nu}\cdot\exp\left\{ \left[\nu\left(1-\gamma\right)-\alpha\right]\ln r_{n}+o\left(-\ln r_{n}\right)\right\} \end{split}
\]
which tends to zero as $n\rightarrow\infty$ whenever $\alpha<\nu\left(1-\gamma\right)$.
So it is sufficient to treat the sum over the off-diagonal terms in
(\ref{eq:expection of energy}). To this end, we follow a similar
approach as the one adopted in the previous step. Again, assume that
$j\neq k$, let $i\in\mathbb{N}$, $0\leq i\leq n-1$, be the unique
integer such that 
\[
2r_{i+1}\leq\left|x_{j}^{(n)}-x_{k}^{(n)}\right|<2r_{i},\quad(\dagger)
\]
and we rewrite the sum over the off-diagonal terms in (\ref{eq:expection of energy})
as 
\begin{equation}
\frac{1}{K_{n}^{2}}\sum_{j=1}^{K_{n}}\sum_{i=0}^{n-1}\sum_{\left\{ k:\,(\dagger)\mbox{ holds with }i\right\} }\frac{\mathcal{W}\left(\Phi_{x_{j}^{(n)},n}\bigcap\Phi_{x_{k}^{(n)},n}\right)}{\mathcal{W}\left(\Phi_{x_{j}^{(n)},n}\right)\mathcal{W}\left(\Phi_{x_{k}^{(n)},n}\right)}\cdot I\left(x_{j}^{(n)},\, x_{k}^{(n)}\right).\label{eq:energy decomposed in i}
\end{equation}
Let $\alpha\in\left(0,\nu\left(1-\gamma\right)\right)$ be fixed.
We investigate the sum in (\ref{eq:energy decomposed in i}) according
to the four cases presented in the previous step. Same as earlier,
without loss of generality, we can assume that $i$ is sufficiently
large.\\

\noindent \emph{Case 1.} Assume that for some $\epsilon\in\left(0,1-\gamma-\frac{\alpha}{\nu}\right)$,
\[
2r_{i+1}\leq\left|x_{j}^{(n)}-x_{k}^{(n)}\right|<r_{i+1}^{1-\varepsilon}.
\]
We have found out in Case 1 previously (following the same arguments
with a possibly smaller $\epsilon$) that 
\[
\begin{split}\frac{\mathcal{W}\left(\Phi_{x_{j}^{(n)},n}\bigcap\Phi_{x_{k}^{(n)},n}\right)}{\mathcal{W}\left(\Phi_{x_{j}^{(n)},n}\right)\mathcal{W}\left(\Phi_{x_{k}^{(n)},n}\right)} & \leq\exp\left[-\nu\gamma\ln r_{i+1}+o\left(-\ln r_{i+1}\right)\right]\end{split}
,
\]
and with $x_{j}^{(n)}$ fixed, the number of $x_{k}^{(n)}$'s that
satisfy the criterion of Case 1 is of the order of $\left(r_{i+1}^{1-\epsilon}/r_{n}\right)^{\nu}$.
Besides, it is easy to see that there exists $C_{\nu}>0$ such that
$I\left(x_{j}^{(n)},\, x_{k}^{(n)}\right)\leq C_{\nu}\cdot r_{i+1}^{-\alpha}$.
So the contribution to (\ref{eq:energy decomposed in i}) under this
case is
\[
\begin{split} & \sum_{i=0}^{n-1}\exp\left\{ \left[\nu\left(1-\epsilon-\gamma\right)-\alpha\right]\ln r_{i+1}+o\left(-\ln r_{i+1}\right)\right\} \end{split}
\]
which is bounded in $n$ since $\epsilon<1-\gamma-\frac{\alpha}{\nu}$.\\

\noindent \emph{Case 2}, \emph{Case 3} and \emph{Case (4b)}. Under
any of the conditions, as imposed in the previous step, of these three
cases, we have that $I\left(x_{j}^{(n)},\, x_{k}^{(n)}\right)\leq C_{\nu}\cdot r_{i}^{-\alpha}$.
Combining this with the findings from the previous step, i.e., the
estimate on 
\[
\begin{split}\frac{\mathcal{W}\left(\Phi_{x_{j}^{(n)},n}\bigcap\Phi_{x_{k}^{(n)},n}\right)}{\mathcal{W}\left(\Phi_{x_{j}^{(n)},n}\right)\mathcal{W}\left(\Phi_{x_{k}^{(n)},n}\right)}\end{split}
\]
and the number of qualifying $k$'s for any fixed $j$, one can easily
confirm that the contribution to (\ref{eq:energy decomposed in i}),
in Case 2 or Case 3 or Case (4b), is bounded in $n$.\\

\noindent \emph{Case (4a)}. However, in the case of (4a), the arguments
above will not work, since $r_{i}^{-\alpha}$ above would be replaced
by $r_{i+1}^{-\alpha\left(1-\epsilon\right)}$. We need to apply a
finer treatment by decomposing the interval $(r_{i+1}^{1-\epsilon},\, r_{i}-r_{i+1}]$
into a union of disjoint intervals. To be specific, let $Z$ be the
smallest integer such that 
\[
\left(r_{i+1}/r_{i}\right)^{\left(1-\gamma\right)^{Z}}\geq1-\frac{r_{i+1}}{r_{i}},
\]
for which to happen it is sufficient to make
\[
\left(1-\gamma\right)^{Z}\leq\frac{\ln\left(1-r_{i+1}/r_{i}\right)}{\ln\left(r_{i+1}/r_{i}\right)},
\]
so $Z$ should be taken as 
\[
\frac{1}{\ln\left(1-\gamma\right)}\ln\left[\frac{\ln\left(1-r_{i+1}/r_{i}\right)}{\ln\left(r_{i+1}/r_{i}\right)}\right]+1=\mathcal{O}\left(-\ln r_{i+1}\right).
\]
Define a sequence of positive numbers $\left\{ R_{m}:\, m=0,\cdots,Z\right\} $
by $R_{0}:=r_{i+1}^{1-\epsilon}$, and 
\[
R_{m}:=r_{i+1}^{\left(1-\gamma\right)^{m}}\cdot r_{i}^{1-\left(1-\gamma\right)^{m}}\mbox{ for }m=1,\cdots,Z.
\]
Clearly, $R_{m}<R_{m+1}$ and 
\[
R_{Z}=r_{i}\cdot\left(r_{i+1}/r_{i}\right)^{\left(1-\gamma\right)^{Z}}\geq r_{i}-r_{i+1}.
\]
Denote $U_{m}:=(R_{m},\, R_{m+1}]$ for $m=0,1,\cdots,Z-1$. Clearly,
\[
(r_{i+1}^{1-\epsilon},\, r_{i}-r_{i+1}]\subseteq\bigcup_{m=0}^{Z-1}U_{m}.
\]
For each $m=0,1,\cdots,Z$, if $\left|x_{j}^{(n)}-x_{k}^{(n)}\right|\in U_{m}$,
then $I\left(x_{j}^{(n)},\, x_{k}^{(n)}\right)\leq C_{\nu}R_{m}^{-\alpha}$.
Recall that, in Case 4,
\[
\frac{\mathcal{W}\left(\Phi_{x_{j}^{(n)},n}\bigcap\Phi_{x_{k}^{(n)},n}\right)}{\mathcal{W}\left(\Phi_{x_{j}^{(n)},n}\right)\mathcal{W}\left(\Phi_{x_{k}^{(n)},n}\right)}\leq\exp\left[-\nu\gamma\ln r_{i}+o\left(-\ln r_{i}\right)\right].
\]
Meanwhile, when $x_{j}^{(n)}$ is fixed, the number of $x_{k}^{(n)}$'s
such that $\left|x_{j}^{(n)}-x_{k}^{(n)}\right|\in U_{m}$ is no greater
than $R_{m+1}^{\nu}/r_{n}^{\nu}$. We will need the following estimate:
\[
\begin{split} & \exp\left[-\nu\gamma\ln r_{i}+o\left(-\ln r_{i}\right)\right]\cdot R_{m}^{-\alpha}\cdot R_{m+1}^{\nu}\\
\leq\; & \exp\left\{ \left[-\nu\gamma-\alpha+\alpha\left(1-\gamma\right)^{m}+\nu-\nu\left(1-\gamma\right)^{m+1}\right]\ln r_{i}+o\left(-\ln r_{i}\right)\right\} \\
 & \qquad\qquad\qquad\qquad\qquad\qquad\cdot\exp\left\{ \left[-\alpha\left(1-\gamma\right)^{m}+\nu\left(1-\gamma\right)^{m+1}\right]\ln r_{i+1}\right\} \\
=\; & \exp\left\{ \left[\nu\left(1-\gamma\right)-\alpha\right]\ln r_{i}+o\left(-\ln r_{i}\right)+\left[\nu\left(1-\gamma\right)-\alpha\right]\left(1-\gamma\right)^{m}\ln\left(r_{i+1}/r_{i}\right)\right\} \\
\leq\; & \exp\left\{ \left[\nu\left(1-\gamma\right)-\alpha\right]\ln r_{i}+o\left(-\ln r_{i}\right)\right\} .
\end{split}
\]
Hence, under the condition (4a), the contribution to (\ref{eq:energy decomposed in i})
is
\[
\begin{split} & \sum_{i=0}^{n-1}\,\sum_{m=0}^{Z-1}\,\exp\left[-\nu\gamma\ln r_{i}+o\left(-\ln r_{i}\right)\right]R_{m}^{-\alpha}\cdot R_{m+1}^{\nu}\\
\leq\; & \sum_{i=0}^{n-1}Z\cdot\exp\left\{ \left[\nu\left(1-\gamma\right)-\alpha\right]\ln r_{i}+o\left(-\ln r_{i}\right)\right\} \\
\leq\; & \sum_{i=0}^{n-1}\exp\left\{ \left[\nu\left(1-\gamma\right)-\alpha\right]\ln r_{i}+o\left(-\ln r_{i}\right)+\mathcal{O}\left(\ln\left(-\ln r_{i+1}\right)\right)\right\} \\
=\; & \sum_{i=0}^{n-1}\exp\left\{ \left[\nu\left(1-\gamma\right)-\alpha\right]\ln r_{i}+o\left(-\ln r_{i}\right)\right\} 
\end{split}
\]
which is bounded in $n$ since $\alpha<\nu\left(1-\gamma\right)$.
The last line is due to (\ref{eq:condition on r_n (2)}).

\subsection*{Step 5: Establish the lower bound.}

From this point on we follow a similar line of arguments as in \cite{HMP}
to complete the proof of the lower bound. Here we outline the key
steps for completeness. Fix any $\alpha\in\left(0,\nu\left(1-\gamma\right)\right)$.
Denote
\[
A_{1}:=\sup_{n\geq1}\,\mathbb{E}^{\mathcal{W}}\left[\left(\mu_{n,\theta}\left(\overline{S\left(O,1\right)}\right)\right)^{2}\right]\mbox{ and }A_{2}:=\sup_{n\geq1}\,\mathbb{E}^{\mathcal{W}}\left[I_{\alpha}\left(\mu_{n,\theta}\right)\right].
\]
For constants $c_{1}>1$, $c_{2}>0$, define the measurable subset
of $\Theta$ 
\[
\Lambda_{n}^{\alpha:}:=\left\{ \theta\in\Theta:\,\frac{1}{c_{1}}\leq\mu_{n,\theta}\left(\overline{S\left(O,1\right)}\right)\leq c_{1},\, I_{\alpha}^{\theta}\left(\mu_{n,\theta}\right)\leq c_{2}\right\} 
\]
and $\Lambda^{\alpha}:=\limsup_{n\rightarrow\infty}\Lambda_{n}^{\alpha}$.
Clearly, 
\[
\sup_{n\geq1}\,\mathcal{W}\left(I_{\alpha}^{\theta}\left(\mu_{n,\theta}\right)>c_{2}\right)\leq\frac{A_{2}}{c_{2}}\mbox{ and }\sup_{n\geq1}\,\mathcal{W}\left(\mu_{n,\theta}\left(\overline{S\left(O,1\right)}\right)>c_{1}\right)\leq\frac{1}{c_{1}}.
\]
Moreover, by (\ref{eq:1st moment of total mass under mu_theta}) and
the Paley-Zygmund inequality, 
\[
\begin{split}\sup_{n\geq1}\,\mathcal{W}\left(\mu_{n,\theta}\left(\overline{S\left(O,1\right)}\right)<\frac{1}{c_{1}}\right) & \leq1-\frac{\left(1-\frac{1}{c_{1}}\right)^{2}}{A_{1}}\end{split}
.
\]
As a consequence, by choosing $c_{1}$ and $c_{2}$ sufficiently large,
we can make 
\[
\mathcal{W}\left(\Lambda_{n}^{\alpha}\right)>\frac{\left(1-\frac{1}{c_{1}}\right)^{2}}{A_{1}}-\frac{1}{c_{1}}-\frac{A_{2}}{c_{2}}>\frac{1}{2A_{1}}
\]
for every $n\geq1$, and hence $\mathcal{W}\left(\Lambda^{\alpha}\right)\geq\frac{1}{2A_{1}}$. 

For every $\theta\in\Lambda^{\alpha}$, there exists a subsequence
$\left\{ n_{k}:k\geq0\right\} $ such that 
\[
\frac{1}{c_{1}}\leq\mu_{n_{k},\theta}\left(\overline{S\left(O,1\right)}\right)\leq c_{1},\, I_{\alpha}\left(\mu_{n_{k},\theta}\right)\leq c_{2}\mbox{ for all }k\geq0.
\]
Because $I_{\text{\ensuremath{\alpha}}}$, as a mapping from the space
of finite measures on $\overline{S\left(O,1\right)}$ to $\left[0,\infty\right]$.
is lower semi-continuous with respect to the weak topology, 
\[
\mathcal{M}:=\left\{ \mu\mbox{ Borel measure on }\overline{S\left(O,1\right)}:\,\frac{1}{c_{1}}\leq\mu\left(\overline{S\left(O,1\right)}\right)\leq c_{1},\, I_{\alpha}\left(\mu\right)\leq c_{2}\right\} 
\]
is compact, and hence for every $\theta\in\Lambda^{\alpha}$, there
exists a Borel measure $\mu_{\theta}$ on $\overline{S\left(O,1\right)}$
such that $\mu_{n_{k},\theta}$ weakly converges to $\mu_{\theta}$
along a subsequence of $\left\{ n_{k}:k\geq0\right\} $. Then, 
\[
\frac{1}{c_{1}}\leq\mu_{\theta}\left(\overline{S\left(O,1\right)}\right)\leq c_{1},\, I_{\alpha}\left(\mu_{\theta}\right)\leq c_{2}.
\]
Moreover, $\Sigma_{\theta}^{\gamma}$, as defined in (\ref{eq:subset of ST}),
is a closed subset of $\overline{S\left(O,1\right)}$. Then the weak
convergence relation implies that $\mu_{\theta}\left(\Sigma_{\theta}^{\gamma}\right)\geq\frac{1}{c_{1}}$.
Therefore, if $\mathcal{C}^{\alpha}\left(\Sigma_{\theta}^{\gamma}\right)$
is the $\alpha-$capacity of the set $\Sigma_{\theta}^{\gamma}$,
i.e., 
\[
\mathcal{C}^{\alpha}\left(\Sigma_{\theta}^{\gamma}\right):=\sup\left\{ \left(\iint_{\Sigma_{\theta}^{\gamma}\times\Sigma_{\theta}^{\gamma}}\frac{\mu\times\mu\left(dydw\right)}{\left|y-w\right|^{\alpha}}\right)^{-1}:\,\mu\mbox{ is a probability measure on }\Sigma_{\theta}^{\gamma}\right\} ,
\]
then $\mathcal{C}^{\alpha}\left(\Sigma_{\theta}^{\gamma}\right)>0$,
and hence, by Frostman's lemma, $\dim_{\mathcal{H}}\left(\Sigma_{\theta}^{\gamma}\right)\geq\alpha$
which implies that $\dim_{\mathcal{H}}\left(ST_{\theta}^{\gamma}\right)\geq\alpha$.
Thus, we have established that
\[
\mathcal{W}\left(\dim_{\mathcal{H}}\left(ST_{\theta}^{\gamma}\right)\geq\alpha\right)\geq\mathcal{W}\left(\Lambda^{\alpha}\right)\geq\frac{1}{2A_{1}}.
\]

Finally, we recall from (\ref{eq:H_basis expansion}) that for $\mathcal{W}-$
a.e. $\theta\in\Theta$, 
\[
\theta=\sum_{n\ge1}\mathcal{I}\left(h_{n}\right)\left(\theta\right)\cdot h_{n}
\]
where $\left\{ h_{n}:n\geq1\right\} $ is an orthonormal basis of
the Cameron-Martin space $H$ and $\left\{ \mathcal{I}\left(h_{n}\right):n\geq1\right\} $
under $\mathcal{W}$ forms a sequence of i.i.d. standard Gaussian
random variables. By a simple application of the Hewitt-Savage zero-one
law, we have that 
\[
\mathcal{W}\left(\dim_{\mathcal{H}}\left(ST_{\theta}^{\gamma}\right)\geq\alpha\right)=1.
\]
Since $\alpha$ is arbitrary in $\left(0,\nu\left(1-\gamma\right)\right)$,
we get the desired lower bound, i.e., 
\[
\dim_{\mathcal{H}}\left(ST_{\theta}^{\gamma}\right)\geq\nu\left(1-\gamma\right)\mbox{ a.s..}
\]
This completes the proof of Theorem \ref{thm:thick point along sequence}.
Since $ST_{\theta}^{\gamma}$ is a subset of $T_{\theta}^{\gamma}$,
we have also established (\ref{eq:lower bound}) and hence Theorem
\ref{thm:hausdorff dimension of thick point set}.

\newpage{}

\bibliographystyle{plain}
\bibliography{mybib}

\begin{thebibliography}{10}

\bibitem{AT}
R.~Adler and J.~Taylor.
\newblock Random fields and geometry.
\newblock {\em Springer}, 2008.

\bibitem{JJRV}
J.~Barral, R.~Rhodes X.~Jin, and V.~Vargas.
\newblock Gaussian multiplicative chaos and kpz duality.
\newblock 2012.
\newblock arXiv:1202.5296v2.

\bibitem{ChatterjeeDemboDing}
S.~Chatterjee, A.~Dembo, and J.~Ding.
\newblock On level sets of gaussian fields.
\newblock pages 1--6, 2013.
\newblock arXiv:1310.5175v1.

\bibitem{CJ}
L.~Chen and D.~Jakobson.
\newblock Gaussian free fields and kpz relation in $\mathbb{R}^4$.
\newblock {\em Annales Henri Poincar\'e}, 15(7):1245--1283, 2014.

\bibitem{add_Gaus}
L.~Chen and D.~Stroock.
\newblock Additive functions and gaussian measures.
\newblock {\em Prokhorov and Contemporary Probability Theory}, 2013.
\newblock Springer Proceedings in Mathematics and Statistics 33.

\bibitem{CiprianiHazra14}
A.~Cipriani and R.~S. Hazra.
\newblock Thick points for gaussian free fields with different cut-offs.
\newblock pages 1--32, 2014.
\newblock arXiv:1407.5840v1.

\bibitem{CiprianiHazra13}
A.~Cipriani and R.~S. Hazra.
\newblock Thick points for a gaussian free fields in 4 dimensions.
\newblock {\em Stochastic Processes and their Applications}, 125(6):2383--2404,
  2015.

\bibitem{DingRoyZeitouni}
J.~Ding, R.~Roy, and O.~Zeitouni.
\newblock Convergence of the centered maximum of log-correlated gaussian
  fields.
\newblock 2015.
\newblock arXiv:1503.04588v1.

\bibitem{DingZeitouni}
J.~Ding and O.~Zeitouni.
\newblock Extreme values for two-dimensional discrete gaussian free field.
\newblock {\em Annals of Probability}, 42(4):1480--1515, 2014.

\bibitem{Dudley}
R.~Dudley.
\newblock Sample functions of the gaussian process.
\newblock {\em Annals of Probability}, 1:66--103, 1973.

\bibitem{DS1}
B.~Duplantier and S.~Sheffield.
\newblock Liouville quantum gravity and kpz.
\newblock {\em Invert. Math.}, 185(2):333--393, 2011.

\bibitem{aws}
L.~Gross.
\newblock Abstract wiener spaces.
\newblock {\em Proc. 5th Berkeley Symp. Math. Stat. and Probab.}, 2(1):31--42,
  1965.

\bibitem{HMP}
X.~Hu, J.~Miller, and Y.~Peres.
\newblock Thick points of the gaussian free field.
\newblock {\em Annals of Probability}, 38(2):896--926, 2010.

\bibitem{Kah}
J.-P. Kahane.
\newblock Random series of functions, 2nd edition.
\newblock {\em Cambridge Studies in Advanced Mathematics 5, Cambridge Univ.
  Press}, 1985.

\bibitem{RV13}
R.~Rhodes and V.~Vargas.
\newblock Gaussian multiplicative chaos and applications: a review.
\newblock 2013.
\newblock arXiv:1305.6221v1.

\bibitem{Shef}
S.~Sheffield.
\newblock Gaussian free fields for mathematicians.
\newblock {\em Probability Theory Related Fields}, 139(3-4):521--541, 2007.

\bibitem{awsrevisited}
D.~Stroock.
\newblock Abstract wiener space, revisited.
\newblock {\em Comm. on Stoch. Anal.}, 2(1):145--151, 2008.

\bibitem{probability}
D.~Stroock.
\newblock Probability, an analytic view, 2nd edition.
\newblock {\em Cambridge Univ. Press}, 2011.

\bibitem{Talagrand}
M.~Talagrand.
\newblock Majorizing measures: The generic chaining.
\newblock {\em Annals of Probability}, 24:1049--1103, 1996.

\bibitem{BesselFunctions}
G.N. Watson.
\newblock A treatise on the theory of bessel functions, 2nd edition.
\newblock {\em Cambridge Univ. Press}, 1995.

\end{thebibliography}

\end{document}